\documentclass[smallcondensed]{svjour3}
\smartqed
\usepackage{graphicx}

\usepackage{amssymb,amsmath, color}
\usepackage{amssymb}
\usepackage{amsbsy}
\usepackage{amsmath}

\usepackage{enumitem,algorithm2e,algorithmic}

\usepackage{ulem}

\DeclareMathOperator*\prox{prox}
\newcommand{\R}{{\mathbb R}}
\newcommand{\N}{{\mathbb N}}
\DeclareMathOperator{\argmin}{argmin}
\newcommand{\interior}{{\rm int}\kern 0.06em}
\newcommand{\inte}{{\rm int}\kern 0.06em}
\newcommand{\cl}{{\rm cl}\kern 0.06em}
\newcommand{\zer}{{\rm zer}\kern 0.06em}
\newcommand{\gph}{{\rm gph}\kern 0.06em}
\newcommand{\dom}{{\rm dom}\kern 0.06em}
\newcommand{\pr}{{\rm pr}\kern 0.06em}
\newcommand{\e}{\epsilon}

\newcommand{\n}{{\nabla}}

\newcommand{\To}{\longrightarrow}
\def\a{\alpha}

\def\b{\beta}
\def\d{\delta}

\def\<{\langle}
\def\>{\rangle}

\newcommand{\cH}{{\mathcal H}}

\newcommand{\demi}{\frac{1}{2}}
\newcommand{\ie}{{\it i.e.}\,\,}

\newcommand{\ds}{\displaystyle}

\usepackage{epsfig}
\usepackage{geometry}
 \usepackage{pict2e}

 \if
 {
 \usepackage[pageref]{backref}
\renewcommand*{\backrefalt}[4]{%
\ifcase #1 %
(Not cited)%
\or
(Cited on p.~#2)%
\else
(Cited on pp.~#2)%
\fi
}

}
\fi

\begin{document}

\title{Convex optimization via inertial algorithms with vanishing Tikhonov regularization: fast convergence to the minimum norm solution}
\titlerunning{Fast convergence to the minimum norm solution}

\author{Hedy Attouch \and Szil\'ard Csaba L\'aszl\'o}
\authorrunning{H. Attouch, L\'aszl\'o} 

\institute{
H. Attouch \at IMAG, Univ. Montpellier, CNRS, Montpellier, France
\\
\email{hedy.attouch@umontpellier.fr} \and
S. L\'aszl\'o
\at Technical University of Cluj-Napoca, Department of Mathematics, Memorandumului 28, Cluj-Napoca, Romania\\ \email{szilard.laszlo@math.utcluj.ro}
}


\maketitle

\begin{abstract}
In a Hilbertian framework, for the minimization of a general convex differentiable function $f$, we introduce new inertial dynamics and algorithms that generate trajectories and iterates that converge fastly towards the minimizer of $f$ with minimum norm.
Our study is based on the non-autonomous version of the Polyak heavy ball method, which, at time $t$,  is associated   with the strongly convex function  obtained by adding to $f$ a Tikhonov regularization term with vanishing coefficient  $\e(t)$. In this dynamic, the damping coefficient is proportional to the square root of the  Tikhonov regularization parameter $\e(t)$.
By adjusting  the speed of convergence of
$\e(t)$ towards zero, we will obtain both  rapid convergence towards the infimal value of $f$, and  the strong convergence of the trajectories towards the element of minimum norm of the set of minimizers of $f$.  In particular, we obtain an improved version  of the dynamic of Su-Boyd-Cand\`es for the accelerated gradient method of Nesterov. This study naturally leads to  corresponding  first-order algorithms obtained by temporal discretization. In the case  of a proper lower semicontinuous and convex function $f$, we study  the  proximal algorithms in detail, and show that they benefit from similar properties.
\end{abstract}


\medskip


\setcounter{tocdepth}{2}


\markboth{H. Attouch, S. C. L\'aszl\'o}{Inertial optimization algorithms with vanishing Tikhonov regularization}


\section{Introduction}
Throughout the paper, $\mathcal H$ is a real Hilbert space which is endowed with the scalar product $\langle \cdot,\cdot\rangle$, with $\|x\|^2= \langle x,x\rangle    $ for  $x\in \mathcal H$.
 We consider
 the convex minimization problem
\begin{equation}\label{edo0001}
 \min \left\lbrace  f (x) : \ x \in \mathcal H \right\rbrace,
\end{equation}
where $f : \mathcal H \rightarrow \mathbb R$ is a convex continuously differentiable function whose solution set $S=\argmin f$ is  nonempty.
We aim at finding  by rapid methods the element of minimum norm of $S$.
As an original aspect of our approach, we start from the Polyak heavy ball with friction dynamic for strongly convex functions, and then adapt it to treat the case of general convex functions. Recall that
a function  $f: \cH \to \mathbb R$ is said to be $\mu$-strongly convex for some $\mu >0$ if   $f- \frac{\mu}{2}\| \cdot\|^2$ is convex.
In this setting, we have the  exponential convergence result:
\begin{theorem}\label{strong-conv-thm}
Suppose that $f: \cH \to \mathbb R$ is a function of class ${\mathcal C}^1$ which is $\mu$-strongly convex for some $\mu >0$.
Let  $x(\cdot): [t_0, + \infty[ \to \cH$ be a solution trajectory of
\begin{equation}\label{dyn-sc-a}
\ddot{x}(t) + 2\sqrt{\mu} \dot{x}(t)  + \nabla f (x(t)) = 0.
\end{equation}
 Then, the following property holds:
$
f(x(t))-  \min_{\mathcal H}f  = \mathcal O \left( e^{-\sqrt{\mu}t}\right)$ \;  as  $t \to +\infty$.
%
\end{theorem}
Let us  see how to take advantage of this fast convergence result, and how to adapt it to the case of a general convex differentiable function $f: \cH \to \mathbb R$. The main idea is linked to Tikhonov's method of regularization. It consists in considering  the corresponding non-autonomous dynamic which at time $t$ is governed by the gradient of the strongly convex function $f_t: \cH \to \mathbb R$
 $$
f_t (x) := f(x) + \frac{\e(t)}{2} \|x\|^2.
 $$
  Then replacing $f$ by $f_t$ in
\eqref{dyn-sc-a}, and noticing that $f_t$ is $\e (t)$-strongly convex, we obtain the dynamic
\begin{equation*}
{\rm(TRIGS)} \qquad \ddot{x}(t) + \d\sqrt{\e(t)}  \dot{x}(t) + \nabla f (x(t)) + \e (t) x(t) =0,
\end{equation*}
 with $\delta=2$. (TRIGS) stands shortly for Tikhonov regularization of inertial gradient systems.
In order not to asymptotically modify the equilibria, we suppose that $\e (t) \to 0$ as $t\to +\infty$. This condition implies that (TRIGS) falls within the framework of the inertial gradient systems with asymptotically vanishing damping.
The importance of this  class of inertial dynamics  has been highlighted by several recent studies
\cite{AAD1}, \cite{ABotCest}, \cite{AC10}, \cite{ACPR}, \cite{AP}, \cite{CD}, \cite{SBC}, which make the link with the accelerated gradient method of Nesterov  \cite{Nest1,Nest2}.

\subsection{Historical facts and related results}

In relation to optimization algorithms, a rich literature has been devoted to the coupling of dynamic gradient systems with Tikhonov regularization.

\subsubsection{First-order gradient dynamics}
 For first-order gradient systems and subdifferential inclusions, the asymptotic hierarchical minimization property which results from the introduction of a vanishing viscosity term in the dynamic (in our context the Tikhonov approximation \cite{Tikh,TA}) has been highlighted in a series of papers  \cite{AlvCab}, \cite{Att2},   \cite{AttCom}, \cite{AttCza2}, \cite{BaiCom}, \cite{CPS}, \cite{Hirstoaga}.
In parallel way, there is  a vast literature on convex descent algorithms involving Tikhonov and more
general penalty, regularization terms. The  historical evolution
 can be traced back to Fiacco and McCormick \cite{FM}, and the interpretation of interior point methods with the help of a vanishing logarithmic barrier.
Some more specific references for the coupling of Prox and Tikhonov can be found in
Cominetti \cite{Com}.
The time discretization  of  the first-order gradient systems and subdifferential inclusions involving multiscale (in time) features provides a natural link between the continuous and discrete  dynamics. The  resulting algorithms  combine  proximal based methods (for example forward-backward algorithms), with the viscosity of penalization methods, see
\cite{AttCzaPey1}, \cite{AttCzaPey2}, \cite{BotCse1}, \cite{Cabot-inertiel,Cab}, \cite{Hirstoaga}.

\medskip

\subsubsection{Second order gradient dynamics}

First studies concerning the coupling of damped inertial dynamics with Tikhonov approximation concerned  the heavy ball with friction system of Polyak \cite{Polyak},
where the damping coefficient $\gamma >0$ is  fixed. In   \cite{AttCza1} Attouch-Czarnecki considered the  system
\begin{equation}\label{HBF-Tikh}
 \ddot{x}(t) + \gamma \dot{x}(t) + \nabla f(x(t)) + \e (t) x(t) =0.
\end{equation}
In the slow parametrization case $\int_0^{+\infty} \e (t) dt = + \infty$, they proved that  any solution $x(\cdot)$ of \eqref{HBF-Tikh} converges strongly to the minimum norm element of $\argmin f$, see also \cite{JM-Tikh}.
 A parallel study has been developed  for PDE's, see \cite{AA} for damped hyperbolic equations with non-isolated equilibria, and
\cite{AlvCab} for semilinear PDE's.
The system \eqref{HBF-Tikh} is a special case of the general dynamic model
\begin{equation}\label{HBF-multiscale}
\ddot{x}(t) + \gamma \dot{x}(t) + \nabla f (x(t)) + \e (t) \nabla g (x(t))=0
\end{equation}
which involves two  functions $f$ and $g$ intervening with different time scale. When $\e (\cdot)$ tends to zero moderately slowly, it was shown in \cite{Att-Czar-last} that the trajectories of (\ref{HBF-multiscale})  converge asymptotically to equilibria that are solutions of the following hierarchical problem: they minimize the function $g$ on the set of minimizers of $f$.
When $\mathcal H= {\mathcal H}_1\times {\mathcal H}_2$ is a product space, defining for $x=(x_1,x_2)$, $f (x_1,x_2):= f_1 (x_1)+f_2 (x_2)$ and   $g(x_1,x_2):= \|A_1 x_1 -A_2 x_2 \|^2$, where the $A_i,\, i\in\{1,2\}$ are linear operators, (\ref{HBF-multiscale})  provides (weakly) coupled inertial systems.
The continuous and discrete-time versions of these systems have a natural connection to the best response dynamics for potential games \cite{AttCza2}, domain decomposition for PDE's \cite{abc2}, optimal transport \cite{abc}, coupled wave equations \cite{HJ2}.

\noindent In the quest for a faster convergence, the following system
\begin{equation}\label{edo001-0}
 \mbox{(AVD)}_{\alpha, \e} \quad \quad \ddot{x}(t) + \frac{\alpha}{t} \dot{x}(t) + \nabla f (x(t)) +\e(t) x(t)=0,
\end{equation}
has been studied by Attouch-Chbani-Riahi \cite{ACR}.
It is a Tikhonov regularization of the  dynamic
\begin{equation}\label{edo001}
 \mbox{(AVD)}_{\alpha} \quad \quad \ddot{x}(t) + \frac{\alpha}{t} \dot{x}(t) + \nabla f (x(t))=0,
\end{equation}
which was introduced by  Su, Boyd and
Cand\`es in \cite{SBC}. When $\alpha =3$, $\mbox{(AVD)}_{\alpha}$  can be viewed as a continuous version of the   accelerated gradient method of Nesterov.
It has been the subject of many recent studies which have given an in-depth understanding of the Nesterov acceleration method, see  \cite{AAD1}, \cite{AC10}, \cite{ACPR}, \cite{SBC}.
 The results obtained in \cite{ACR} concerning \eqref{edo001-0} will serve as a basis for comparison.

\subsection{Model results}
To illustrate our results, let us consider the case $\e (t) = \frac{c}{t^r}$ where $r$ is positive parameter satisfying $0<r \leq 2$. The case $r=2$ is of particular interest, it is related to the continuous version of the accelerated gradient method of Nesterov, with optimal convergence rate for general convex differentiable function $f$.

   \subsubsection{Case $r=2$}
   Let us consider the (TRIGS) dynamic
\begin{equation}\label{DynSys-crit-model}
\ddot{x}(t) + \frac{\a}{t} \dot{x}(t) + \nabla f\left(x(t) \right)+\frac{c}{t^2}x(t)=0,
\end{equation}
where the parameter $\alpha \geq 3$ plays a crucial role.
As a consequence of Theorems  \ref{RateConvergenceResult} and \ref{StrongConvergence} we have

\begin{theorem}\label{RateConvergenceResult-model}
Let $x : [t_0, +\infty[ \to \mathcal{H}$ be a solution of  \eqref{DynSys-crit-model}.
We then have the following results:

$i)$ If $\a=3$, then
$\displaystyle f\left(x(t)\right) - \min_{\cH} f =O\left(\frac{\ln t}{t^2}\right) \mbox{ as }t\to+\infty.$

$ii)$ If $\a>3$, then
$\displaystyle f\left(x(t)\right) - \min_{\cH} f =O\left(\frac{1}{t^2}\right) \mbox{ as }t\to+\infty.$ Further, the trajectory $x$ is bounded,
$\displaystyle \|\dot{x}(t)\| =O\left(\frac{1}{t}\right) \mbox{ as }t\to+\infty$, and
there is  strong convergence to the minimum norm solution:
$$ \liminf_{t \to +\infty}{\| x(t) - x^\ast \|} = 0.$$
\end{theorem}

 \subsubsection{Case $r<2$}

As a consequence of Theorems \ref{RateConvergenceResult-bb} and \ref{StrongRateConvergenceResult-bb}, we have:

\begin{theorem}\label{thm:model}
Take $\e(t)=1/t^r$, $  \frac{2}{3} <r<2$.
Let $x : [t_0, +\infty[ \to \mathcal{H}$ be a global solution trajectory of
$$
\ddot{x}(t) +  \frac{\d}{t^{\frac{r}{2}}}\dot{x}(t) + \nabla f\left(x(t) \right)+ \frac{1}{t^{r}} x(t)=0.
$$
Then, we have fast convergence the values, and strong convergence to the minimum norm solution:
$$
f(x(t))-\min_{\cH} f= \mathcal O \left( \displaystyle{ \frac{1}{t^{\frac{3r}{2}-1}} }   \right) \mbox{ and } \liminf_{t \to + \infty}{\| x(t) - x^\ast \|} = 0.
$$
\end{theorem}
These results are completed by showing that, if there exists $T \geq t_0$, such that the trajectory $\{ x(t) :t \geq T \}$ stays either in the open ball $B(0, \| x^\ast \|)$ or in its complement, then $x(t)$ converges strongly to $x^*$ as $t\to+\infty.$
Corresponding results for the associated proximal algorithms, obtained by temporal discretization, are obtained in Section \ref{intro-dyn-RIPA}.

A remarkable property of the above results is that the rate of convergence of values is comparable to  the Nesterov accelerated gradient method. In addition, we have a strong convergence property to the  minimum norm solution,  with  comparable numerical complexity.
These results represent an important advance compared to previous works by producing new dynamics for which we have both  rapid convergence of values and strong convergence towards the solution of minimum norm.
Let us stress the fact that in our approach the fast convergence of the values and the strong convergence towards the solution of  minimum norm are obtained for the same dynamic, whereas in the previous works \cite{ACR}, \cite{AttCza1}, they are obtained for different dynamics obtained for different settings of the parameters.
It is clear that the results extend naturally to obtaining  strong convergence towards the solution closest to a desired state $x_d$.
It suffices to replace in Tikhonov's approximation $\|x\|^2$ by
$\|x-x_d\|^2$. This is important for inverse problems.

\subsection{Contents}
 In section \ref{sec:general}, we show existence and uniqueness of a  global solution for  the Cauchy problem associated with (TRIGS).
Then,  based on  Lyapunov analysis, we obtain convergence rates of the values  which are valid for a general $\e (\cdot)$.
Section \ref{sec:critical} is  devoted to an
in-depth analysis in the  critical case $\e(t) = c/t^2$.
Section \ref{sec:strong} is devoted to the study of the strong convergence property of the trajectories towards the minimum norm solution, in the case of a general $\e(\cdot)$.
Then in Section \ref{intro-dyn-RIPA} we obtain similar results  for the associated proximal algorithms, obtained by temporal discretization.

\section{Convergence analysis for general $\e(t)$}\label{sec:general}
We are going to analyze via Lyapunov analysis the convergence properties as $t\to +\infty$ of the solution trajectories of the inertial dynamic (TRIGS) that we recall below
\begin{equation}\label{TRIGS}
 \ddot{x}(t) + \d\sqrt{\e(t)}  \dot{x}(t) + \nabla f (x(t)) + \e (t) x(t) =0.
\end{equation}

Throughout the paper, we assume that $t_0$ is the origin of time, $\delta$ is a positive parameter, and

\begin{description}

\item[ $(H_1)$ ]   $f : \mathcal H \rightarrow \mathbb R$ is convex and differentiable,  $\nabla f$ is Lipschitz continuous on  bounded sets.

\smallskip

\item[ $(H_2)$ ]  $S := \mbox{argmin} f \neq \emptyset$. We denote by $x^*$ the element of minimum norm of $S$.

\smallskip

\item[ $(H_3)$ ]  $\e : [t_0 , +\infty [ \to \mathbb R^+ $ is   a nonincreasing function, of class $\mathcal C^1$, such that \  $\lim_{t \to \infty}\e (t) =0$.
\end{description}

\subsection{Existence and uniqueness for the Cauchy problem}
Let us first show that the Cauchy problem for (TRIGS) is well posed.

\begin{theorem}\label{Cauchy-weel-posed}
Given $(x_0, v_0) \in \cH \times \cH$, there exists a unique global classical solution $x : [t_0, +\infty[ \to \mathcal{H}$ of the Cauchy problem
\begin{align}\label{DynSys-1}
\begin{cases}
\ddot{x}(t) + \d\sqrt{\e(t)} \dot{x}(t) + \nabla f\left(x(t) \right)+\e (t) x(t)=0 \vspace{2mm}\\
x(t_0) = x_0, \,
\dot{x}(t_0) = v_0.
\end{cases}
\end{align}
\end{theorem}

\begin{proof} The proof  relies on the combination of the Cauchy-Lipschitz theorem with energy estimates.
First consider the Hamiltonian formulation of  (\ref{DynSys-1}) as the first order system
\begin{align}\label{DynSys-2}
\begin{cases}
\dot{x}(t) - y(t) =0 \vspace{1mm} \\
\dot{y}(t) + \d\sqrt{\e(t)} y(t) + \nabla f\left(x(t) \right)+\e (t) x(t)=0 \vspace{1mm}\\
x(t_0) = x_0, \,
y(t_0) = v_0.
\end{cases}
\end{align}
According to the hypothesis $(H_1), (H_2), (H_3)$, and by applying the Cauchy-Lipschitz theorem in the locally Lipschitz case, we obtain the existence and uniqueness of a local solution.
Then, in order to pass from a local solution to a global solution, we rely on the  energy estimate obtained by taking the scalar product of (TRIGS) with $\dot{x}(t)$. It gives
$$
\frac{d}{dt} \Big( \demi \| \dot{x}(t)\|^2 + f(x(t)) + \demi \e (t) \|x(t)\|^2 )  \Big)
+ \d\sqrt{\e(t)} \| \dot{x}(t)\|^2 - \demi \dot{\e} (t) \|x(t)\|^2 =0.
$$
From $(H_3)$, $\e(\cdot)$ is non-increasing. Therefore, the energy function $t \mapsto W(t)$ is decreasing where
$$
W(t):= \demi \| \dot{x}(t)\|^2 + f(x(t)) + \demi \e (t) \|x(t)\|^2 .
$$
The end of the proof follows a standard argument. Take a maximal solution defined on an interval $[t_0, T[$. If $T$ is infinite, the
proof is over. Otherwise, if $T$ is finite, according to the above energy estimate, we have that $\| \dot{x}(t)\|$ remains bounded, just like $\| x(t)\|$ and $\| \ddot{x}(t)\|$ (use (TRIGS)). Therefore, the limit of $x(t)$ and  $\dot{x}(t)$ exists when $t \to T$. Applying the local existence result  at $T$ with the  initial conditions thus obtained gives a contradiction to the maximality of the solution.
\end{proof}

\subsection{General case}
The control of the decay of $\e(t)$ to zero as $t \to +\infty$ will play a key role in the  Lyapunov analysis of (TRIGS).
Precisely, we will use the following condition.

\begin{definition}
Given $\delta >0$, we say that $t \mapsto \e(t)$ satisfies the controlled decay property ${\rm(CD)}_K$, if it is a nonincreasing function which satisfies: there exists $t_1\ge t_0$ such that for all $t\ge t_1,$
$$\left(\frac{1}{\sqrt{\e(t)}}\right)'\le \min(2K-\d, \d-K),$$ where $K$ is a  parameter such that $ \frac{\d}{2} < K  < \delta$ for $0<\delta \leq 2$,  and $ \frac{\d+ \sqrt{\delta^2 -4}}{2} < K  < \delta$ for $\delta > 2$ .
\end{definition}

\begin{theorem}\label{RateConvergenceResult-a}
Let $x : [t_0, +\infty[ \to \mathcal{H}$ be a solution trajectory of  {\rm(TRIGS)}. Let $\delta$ be a positive parameter.
Suppose that $\e(\cdot)$ satisfies the condition ${\rm(CD)}_K$ for some $K>0$.
Then,  we have the following rate of convergence of values: for all $t\ge t_1$
\begin{equation}\label{basic-Lyap}
f(x(t)) - \min_{\cH} f \le \frac{K\|x^*\|^2}{2}\frac{1}{\mathfrak{M}(t)}\int_{t_1}^t \e^{\frac32}(s)\mathfrak{M}(s)ds+\frac{C}{\mathfrak{M}(t)},
\end{equation}
where
\begin{equation*}
\mathfrak{M}(t)=\exp \left({\displaystyle{\int_{t_1}^t\mu(s)ds}}\right), \quad
\mu(t) =-\frac{\dot{\e}(t)}{2\e(t)}+ (\delta-K)\sqrt{\e(t)}
\end{equation*}
and $$C=\left( f(x(t_1)) - f(x^\ast) \right) +\frac{\e(t_1)}{2}\|x(t_1)\|^2+ \frac{1}{2} \| K\sqrt{\e(t_1)}(x(t_1)-x^\ast) + \dot{x}(t_1) \|^2.$$
\end{theorem}

\begin{proof}
{\bf Lyapunov analysis.} Set $f^\ast := f(x^\ast)=\min_{\cH} f$. The  energy function $\mathcal{E} : [t_0, +\infty[ \to \mathbb{R}_+,$
\begin{align}\label{Lyapunov-a}
\mathcal{E}(t):= \left( f(x(t)) - f^\ast \right) +\frac{\e(t)}{2}\|x(t)\|^2+ \frac{1}{2} \| c(t)(x(t)-x^\ast) + \dot{x}(t) \|^2,
\end{align}
will be the basis for our Lyapunov analysis.
The function $c:[t_0,+\infty[\to\R$ will be defined later, appropriately.
Let us  differentiate $\mathcal{E}(\cdot)$. By  using the derivation chain rule, we get
\begin{align}\label{engderiv1-a}
\dot{\mathcal{E}}(t) &= \<\n f(x(t)),\dot{x}(t)\>+\frac{\dot{\e}(t)}{2}\|x(t)\|^2+\e(t)\<\dot{x}(t),x(t)\>\\
\nonumber&+\<c'(t)(x(t)-x^*)+c(t)\dot{x}(t) +\ddot{x}(t),c(t)(x(t)-x^\ast) +  \dot{x}(t)\>.
\end{align}
According to the constitutive equation  (\ref{TRIGS}), we have
\begin{align}\label{SecondOrderDeriv-a}
\ddot{x}(t) = - \e(t) x(t)-\d\sqrt{\e(t)} \dot{x}(t) - \nabla f(x(t)).
\end{align}
Therefore,
\begin{eqnarray}\label{forenergy2-a}
&&\<c'(t)(x(t)-x^*)+c(t)\dot{x}(t) +\ddot{x}(t),c(t)(x(t)-x^\ast) + \dot{x}(t)\> \\
&&= \<c'(t)(x(t)-x^*)+(c(t)-\d\sqrt{\e(t)})\dot{x}(t) -(\e(t)x(t)+\nabla f(x(t))), c(t)(x(t)-x^\ast) +  \dot{x}(t)\> \nonumber\\
&& =c'(t)c(t)\|x(t)-x^*\|^2+(c'(t)+c^2(t)-\d c(t)\sqrt{\e(t)})\<\dot{x}(t),x(t)-x^*\>+(c(t)-\d\sqrt{\e(t)})\|\dot{x}(t)\|^2 \nonumber\\
&&-\e(t)\<x(t),\dot{x}(t)\> -\<\nabla f(x(t)), \dot{x}(t)\>-c(t)\<\e(t)x(t)+\nabla f(x(t)), x(t)-x^\ast\>. \nonumber
\end{eqnarray}
By combining \eqref{engderiv1-a} with \eqref{forenergy2-a}, we get
\begin{align}\label{engderiv2-a}
\dot{\mathcal{E}}(t) =&\frac{\dot{\e}(t)}{2}\|x(t)\|^2+c'(t)c(t)\|x(t)-x^*\|^2+(c'(t)+c^2(t)-\d c(t)\sqrt{\e(t)})\<\dot{x}(t),x(t)-x^*\>\\
\nonumber& +(c(t)-\d\sqrt{\e(t)})\|\dot{x}(t)\|^2-c(t)\<\e(t)x(t)+\nabla f(x(t)), x(t)-x^\ast\> .
\end{align}
Consider  the function
$$f_t:\mathcal{H}\To\R,\,f_t(x)=f(x)+\frac{\e(t)}{2}\|x\|^2.$$
According to the strong convexity property of $f_t$, we have
$$f_t(y)-f_t(x)\ge\<\n f_t(x),y-x\>+\frac{\e(t)}{2}\|x-y\|^2,\mbox{ for all }x,y\in\mathcal{H}.$$
Take  $y=x^*$ and $x=x(t)$ in the above inequality. We get
\begin{align*}
&f(x^*)+\frac{\e(t)}{2}\|x^*\|^2-f(x(t))-\frac{\e(t)}{2}\|x(t)\|^2\ge\\
&-\<\n f(x(t))+\e(t)(x(t),x(t)-x^*\>+\frac{\e(t)}{2}\|x(t)-x^*\|^2.
\end{align*}
Consequently,
\begin{eqnarray}
-\<\n f(x(t))+\e(t)x(t),x(t)-x^*\> &\leq & -(f(x(t))-f(x^*)) \nonumber \\
&+&\frac{\e(t)}{2}\|x^*\|^2-\frac{\e(t)}{2}\|x(t)\|^2-\frac{\e(t)}{2}\|x(t)-x^*\|^2.\label{forenergy4-a}
\end{eqnarray}
By multiplying \eqref{forenergy4-a} with $c(t)$ and injecting in \eqref{engderiv2-a} we get
\begin{align}\label{engderiv3-a}
\dot{\mathcal{E}}(t)\le&-c(t)(f(x(t))-f^*)+\left(\frac{\dot{\e}(t)}{2}-c(t)\frac{\e(t)}{2}\right)\|x(t)\|^2\\
 \nonumber&+\left(c'(t)c(t)-c(t)\frac{\e(t)}{2}\right)\|x(t)-x^*\|^2+(c(t)-\d\sqrt{\e(t)})\|\dot{x}(t)\|^2\\
 \nonumber&+(c'(t)+c^2(t)-\d c(t)\sqrt{\e(t)})\<\dot{x}(t),x(t)-x^*\>+c(t)\frac{\e(t)}{2}\|x^*\|^2.
\end{align}
On the other hand, for a positive function $\mu(t)$ we have
\begin{align}\label{energ1-a}
\mu(t)\mathcal{E}(t)=&\mu(t) \left( f(x(t)) - f^\ast \right) +\mu(t)\frac{\e(t)}{2}\|x(t)\|^2+ \frac{1}{2}\mu(t)c^2(t) \|x(t)-x^\ast\|^2 +\frac{1}{2}\mu(t) \| \dot{x}(t) \|^2\\
\nonumber &+ \mu(t)c(t)\<\dot{x}(t),x(t)-x^*\>.
\end{align}
By adding \eqref{engderiv3-a} and \eqref{energ1-a} we get
\begin{align}\label{forgronwal1}
\dot{\mathcal{E}}(t)+\mu(t)\mathcal{E}(t)\le &(\mu(t)-c(t))( f(x(t)) - f^\ast)+\left(\frac{\dot{\e}(t)}{2}-c(t)\frac{\e(t)}{2}+\mu(t)\frac{\e(t)}{2}\right)\|x(t)\|^2\\
\nonumber& +\left(c'(t)c(t)-c(t)\frac{\e(t)}{2}+\frac{1}{2}\mu(t)c^2(t)\right)\|x(t)-x^*\|^2\\
\nonumber&+\left(c(t)-\d\sqrt{\e(t)}+\frac{1}{2}\mu(t)\right)\|\dot{x}(t)\|^2\\
 \nonumber&+\left(c'(t)+c^2(t)-\d c(t)\sqrt{\e(t)}+ \mu(t)c(t)\right)\<\dot{x}(t),x(t)-x^*\>+c(t)\frac{\e(t)}{2}\|x^*\|^2.
\end{align}
Since we have no control on the sign of $\<\dot{x}(t),x(t)-x^*\>$, we take the coefficient in front of this term equal to zero, that is
\begin{equation}\label{def:c}
c'(t)+c^2(t)-\d c(t)\sqrt{\e(t)}+ \mu(t)c(t) =0.
\end{equation}
Take  $c(t)=K\sqrt{\e(t)}$. Indeed, it is here that the choice of $c$, and of the corresponding parameter $K$,  come into play.  The relation (\ref{def:c}) can be equivalently written
$$\mu(t) =-\frac{\dot{\e}(t)}{2\e(t)}+ (\delta-K)\sqrt{\e(t)}.$$
According to this choice for $\mu(t)$ and $c(t)$, the inequality
(\ref{forgronwal1}) becomes
\begin{align}\label{forgronwal1-b}
\dot{\mathcal{E}}(t)+\mu(t)\mathcal{E}(t)\le &\frac{1}{2\e(t)}\left(-\dot{\e}(t)+2(\d-2K){\e(t)}^{\frac32}\right)( f(x(t)) - f^\ast)\\
\nonumber&+\frac{1}{4}\left(  \dot{\e}(t) +  2\left( \delta -2K \right)   \e(t)^{\frac32}\right)\|x(t)\|^2\\
\nonumber& +\frac{K}{4}\left( K\dot{\e}(t)  +2\e(t)^{\frac32}(-K^2+ \delta K -1) \right)\|x(t)-x^*\|^2\\
\nonumber&+ \frac{1}{4\e(t)} \left( -\dot{\e}(t)+2(K- \delta){\e(t)}^{\frac32}   \right)\|\dot{x}(t)\|^2  +\frac{K\|x^*\|^2}{2}\e^{\frac32}(t).
\end{align}
Let us show that  the condition ${\rm(CD)}_K$ provide the  nonpositive sign for the coefficients in front of the  terms of the right side of (\ref{forgronwal1-b}).
Recall that, according to the hypotheses ${\rm(CD)}_K$, for all $t\ge t_1$ we have  the  properties \textit{a)} and \textit{b)}:
\begin{eqnarray*}
&& a) \; \left(\frac{1}{\sqrt{\e(t)}}\right)'\le M_1 (K)= \min(2K-\d, \d-K)=\left\{\begin{array}{ll}2K-\d\mbox{ if } K\le \frac23 \d \\\d-K,\mbox{ if }\frac23 \d\le K,\end{array}\right. \\
&& b) \; \left(\frac{1}{\sqrt{\e(t)}}\right)'\ge0.
\end{eqnarray*}
Without ambiguity we write briefly $M_1$ for $M_1 (K)$.
Note that $b)$ just expresses that $\e(\cdot)$ is non increasing.
According to the hypotheses ${\rm(CD)}_K$, we claim that
for all $t\ge t_1$
\begin{equation}\label{difineqsysforc1}
\left\{
\begin{array}{llll}
i) \; \left(\frac{1}{\sqrt{\e(t)}}\right)'\le 2K-\d\\
ii) \; \left(\frac{1}{\sqrt{\e(t)}}\right)'\ge \frac{\d K-K^2-1}{K}
\\
iii) \; \left(\frac{1}{\sqrt{\e(t)}}\right)'\le \d-K.
\end{array}\right.
\end{equation}
Let us justify  these inequalities  (\ref{difineqsysforc1}).\\
$i)$  is a consequence of $\left(\frac{1}{\sqrt{\e(t)}}\right)'\le M_1$ and $M_1\leq 2K-\delta$.\\
$ii)$  is a consequence of $\left(\frac{1}{\sqrt{\e(t)}}\right)'\ge0$ and  $\d K-K^2-1 \leq 0$. Precisely, when  $\delta \leq 2$ we have  $\d K-K^2-1  \leq 2 K-K^2-1 \leq 0$. When $\delta > 2$, we have   $\d K-K^2-1 \leq 0$ because  $K \geq \frac{\d+ \sqrt{\delta^2 -4}}{2} $.\\
$iii)$  is a consequence of $\left(\frac{1}{\sqrt{\e(t)}}\right)'\le M_1$ and $M_1\leq \delta -K$.

\smallskip

\noindent The inequalities (\ref{difineqsysforc1}) can be equivalently written as follows:   for all $t\ge t_1$
\begin{equation}\label{difineqsysforc}
\left\{
\begin{array}{llll}
i) \; -\dot{\e}(t)+2(\d-2K){\e(t)}^{\frac32}\le 0 \vspace{2mm}\\
ii) \; K\dot{\e}(t)+2(\d K-K^2-1){\e(t)}^{\frac32}\le 0\vspace{2mm}\\
iii) \; -\dot{\e}(t)+2(K-\d){\e(t)}^{\frac32}\le 0.\vspace{2mm}
\end{array}\right.
\end{equation}
The inequalities (\ref{difineqsysforc})  give that the coefficients
entering the right side of (\ref{forgronwal1-b}) are  nonpositive:

\smallskip

\noindent $\bullet$ $i)$ gives that the coefficient of  $f(x(t)) - f^\ast$  is nonpositive.  \\
$\bullet$  Since $\dot{\e}(t)\le 0$ we have $\dot{\e}(t)+2(\d -2K){\e(t)}^{\frac32}\le  -\dot{\e}(t)+2(\d -2K){\e(t)}^{\frac32}$. Therefore, by $i)$
we have
that the coefficient of $\|x(t)\|^2$ in \eqref{forgronwal1-b} is  nonpositive.\\
$\bullet$ $ii)$ gives that the coefficient of $\|x(t)-x^*\|^2$  is nonpositive.\\
$\bullet$ $iii)$  gives that the coefficient of $\|\dot{x}(t)\|^2 $
is  nonpositive .

\smallskip

\noindent Let us return to (\ref{forgronwal1-b}). Using (\ref{difineqsysforc}) and the above results, we obtain
\begin{equation}\label{Basic1}
\dot{\mathcal{E}}(t)+\mu(t)\mathcal{E}(t)\le\frac{K\|x^*\|^2}{2}\e^{\frac32}(t),\mbox{ for all }t\ge t_1.
\end{equation}
By multiplying \eqref{Basic1} with $\mathfrak{M}(t)= \exp \left({\displaystyle{\int_{t_1}^t\mu(s)ds}}\right)$  we obtain
\begin{equation}\label{forgronwal2}
\frac{d}{dt}\left(\mathfrak{M}(t)\mathcal{E}(t)\right)\le \frac{K\|x^*\|^2}{2}\e^{\frac32}(t)\mathfrak{M}(t).
\end{equation}
By integrating \eqref{forgronwal2} on $[t_1,t]$ we get
\begin{equation}\label{energyrate}
\mathcal{E}(t)\le \frac{K\|x^*\|^2}{2}\frac{\int_{t_1}^t \e^{\frac32}(s)\mathfrak{M}(s)ds}{\mathfrak{M}(t)}+\frac{\mathfrak{M}(t_1)\mathcal{E}(t_1)}{\mathfrak{M}(t)}.
\end{equation}
By definition of $\mathcal{E}(t)$ we deduce that
\begin{equation}\label{valuerate}
f(x(t)) - \min_{\cH} f \le \frac{K\|x^*\|^2}{2}\frac{\int_{t_1}^t \e^{\frac32}(s)\mathfrak{M}(s)ds}{\mathfrak{M}(t)}+\frac{\mathcal{E}(t_1)}{\mathfrak{M}(t)},
\end{equation}
for all $t\ge t_1$, and this gives  the  convergence rate of the values.
\end{proof}
\begin{remark}\label{formepsilon}
By integrating the  relation $0\le\left(\frac{1}{\sqrt{\e(t)}}\right)'\le M_1$ on an interval $[t_1, t]$, we get
$$\frac{1}{\sqrt{\e(t_1)}}\le \frac{1}{\sqrt{\e(t)}}\le M_1 t+\frac{1}{\sqrt{\e(t_1)}}-M_1 t_1 .$$
Therefore, denoting $C_1=\frac{1}{\sqrt{\e(t_1)}}-M_1 t_1$, and $C_2=\e(t_1)$ we have
\begin{equation}\label{formeps}
\frac{1}{(M_1 t+C_1)^2} \leq \e(t) \leq C_2.
\end{equation}
This shows that the Lyapunov analysis developed previously only provides information  in the case where
   $ \e (t) $ is greater than or equal to $ C / t^2 $.  Since the damping coefficient  $\gamma(t) = \d\sqrt{\e(t)}$, this means that $\gamma (t)$ must be greater than or equal to $C/t$.
 This is in accordance with the theory of inertial gradient systems with time-dependent viscosity coefficient, which states that the asymptotic optimization property is valid provided that the integral on $[t_0, +\infty[$ of $\gamma(t)$ is infinite,  see \cite{AC10}.
\end{remark}

As a consequence of Theorem \ref{RateConvergenceResult-a} we have  the following result.

\begin{corollary}\label{gen_cor_1}
Under the hypothesis of Theorem  \ref{RateConvergenceResult-a} we have
\begin{equation}\label{rem_min_1}
\lim_{t\to+\infty}\mathfrak{M}(t)=+\infty.
\end{equation}
Suppose moreover  that $\e^{\frac32}(\cdot)\in L^1(t_0,+\infty)$. Then
\begin{equation}\label{rem_min_2}
\lim_{t\to+\infty}f(x(t))=\min_{\cH} f.
\end{equation}
\end{corollary}
\begin{proof}
By definition of $\mu(t)$, since $\e(\cdot)$ is nonincreasing and $ \delta \geq K$, we have that
$\mu (t)$ is nonnegative for all $t\geq t_1$. Therefore,
$t \mapsto \mathfrak{M}(t)$ is a nondecreasing function. Let us write
equivalently $\mu(t)=\frac{d}{dt}\ln \frac{1}{\sqrt{\e(t)}}+ (\delta-K)\sqrt{\e(t)}$, and integrate  on $[t_1,t]$. We obtain
$$\mathfrak{M}(t)= \exp \left({\displaystyle{\int_{t_1}^t\mu(s)ds}}\right)=\frac{C}{\sqrt{\e(t)}}\exp\left(\int_{t_1}^t (\delta-K)\sqrt{\e(s)}ds\right).$$
Since $ \delta -K\geq 0$, we deduce that  $\mathfrak{M}(t) \geq \frac{C}{\sqrt{\e(t)}}$. Since $\lim_{t \to \infty}\e (t) =0$,
we get
\begin{equation*}
\lim_{t\to+\infty}\mathfrak{M}(t)=+\infty.
\end{equation*}
Moreover, if we suppose that $\e^{\frac32}(\cdot)\in L^1(t_0,+\infty)$, then  by \cite[Lemma A.3]{ACR} we obtain
\begin{equation*}
\lim_{t\to+\infty} \frac{\int_{t_1}^t \e^{\frac32}(s)\mathfrak{M}(s)ds}{\mathfrak{M}(t)}=0.
\end{equation*}
Combining these properties with the convergence rate (\ref{basic-Lyap}) of Theorem \ref{RateConvergenceResult-a}, we obtain (\ref{rem_min_2}).
\end{proof}

\subsection{Particular cases}

Since $\e (t) \to 0$ as $t \to + \infty$,  (TRIGS)  falls within the setting of the inertial dynamics with an asymptotic vanishing damping coefficient $\gamma (t)$.  Here, $\gamma (t) = \d\sqrt{\e(t)} $. We know with Cabot-Engler-Gaddat \cite{CEG} that for such systems,  the optimization property is  satisfied asymptotically if $\int_{t_0}^{+\infty} \gamma(t) dt =+\infty$ (\ie $\gamma (t)$ does no tend too rapidly towards zero). By taking $\e (t) = \frac{c}{t^p}$, it is easy to verify
that the condition ${\rm(CD)}_K$ is satisfied if $p \leq 2$, that is $\sqrt{\e(t)} = \frac{c}{t^p}$, with $p\leq 1$,    which is in accordance with the above property.
Let us  particularize  Theorem  \ref{RateConvergenceResult-a} to situations where the integrals can be computed (at least estimated).

\subsubsection{$\e(t) $ of order $1/t^2$}
Take
$$\e(t)=\frac{1}{(M t+C)^2},\, \; M < M_1 (K),\ \; C\le C_1 .$$
 Then, $\left(\frac{1}{\sqrt{\e(t)}}\right)'\le M_1 (K)$ for all $t\ge t_0$ and the condition ${\rm(CD)}_K$  is satisfied. Moreover,
 $$\mu(t)=\frac{M+\d-K}{Mt+C}, \quad
\mathfrak{M}(t)=\left(\frac{Mt+C}{Mt_0+C}\right)^{\frac{M+\d-K}{M}}.
$$
Therefore, \eqref{basic-Lyap}  becomes
\begin{equation}\label{energyrate1}
\mathcal{E}(t)\le \frac{K\|x^*\|^2}{2}\frac{\displaystyle{\int_{t_0}^t {(Ms+C)}^{\frac{-2M+\d-K}{M}}ds}}{{(Mt+C)}^{\frac{M+\d-K}{M}}}+\frac{(Mt_0+C)^{\frac{M+\d-K}{M}}\mathcal{E}(t_0)}{(Mt+C)^{\frac{M+\d-K}{M}}}.
\end{equation}
Consequently, we have
\begin{equation*}
\mathcal{E}(t)\le \frac{K\|x^*\|^2}{2(-M+\d-K)}\frac{1}{{(Mt+C)}^2}+\frac{-\frac{K\|x^*\|^2}{2(-M+\d-K)}(Mt_0+C)^{\frac{-M+\d-K}{M}}+(Mt_0+C)^{\frac{M+\d-K}{M}}\mathcal{E}(t_0)}{(Mt+C)^{\frac{M+\d-K}{M}}}.
\end{equation*}
By assumption we have $M < M_1 \leq \delta -K$. Therefore $\frac{M+\d-K}{M} > 2$ and $-M +\delta -K > 0$. It follows that when $Mt +C\geq 1$
$$\mathcal{E}(t)\le \frac{C'}{(Mt+C)^2}, \;
 \mbox{ with } \; \displaystyle{C'=\frac{K\|x^*\|^2}{2(-M+\d-K)}+(Mt_0+C)^{\frac{M+\d-K}{M}}\mathcal{E}(t_0)}.$$
\noindent Observe that $\displaystyle{\d\sqrt{\e(t)}=\frac{\frac{\d}{M}}{t+\frac{C}{M}}=\frac{\a}{t+\b}}$, where we set
$\alpha = \frac{\d}{M}$ and $\beta =\frac{C}{M}$. Since  $M<  M_1\le\frac13 \d$ we get  $\a\in\left]3,+\infty \right[$. Indeed, we can get any $\alpha > 3$.
Note also that by translating the time scale the result in the general case $ \beta \geq 0 $ results from its obtaining for a particular case $ \beta=0$.
According to the fact that we can take for $\delta$ any positive number, we obtain

\begin{theorem}\label{RateConvergenceResult-b}
Take $\a\in\left]3,+\infty\right[,$ $c >0$.
Let $x : [t_0, +\infty[ \to \mathcal{H}$ be a  solution trajectory of
$$
\ddot{x}(t) +  \frac{\a}{t} \dot{x}(t) + \nabla f\left(x(t) \right)+\frac{c}{t^2}x(t)=0.
$$
Then, the following convergence rate of the values is satisfied: as $t\to +\infty$
$$
f(x(t))-\min_{\cH} f=O\left(\frac{1}{t^2}\right).
$$
\end{theorem}

\begin{remark}\label{Tikh_interest1}
It is an natural question to compare our dynamic ($c>0$) with
the Su-Boyd-Cand\`es  dynamic \cite{SBC} ($c=0$), which was introduced as a continuous version of the Nesterov accelerated gradient method. We obtain the  optimal  convergence rate of  values with an additional Tikhonov regularization term, which is a remarkable property.
 In fact, in the next sections we will prove that the Tikhonov term induces  strong convergence of the trajectory to the minimum norm solution.
 \end{remark}

 \subsubsection{$\e(t)$ of order $1/t^r$, $\frac{2}{3}< r<2$}
Take $\e(t)=1/t^r$, $r<2$.
Then
\begin{eqnarray*}
\mu(t)& =&-\frac{1}{2}\frac{\dot{\e}(t)}{\e(t)}+  (\d-K)\sqrt{\e(t)}\\
& =&\frac{r}{2t}+  \frac{\d-K}{ t^\frac{r}{2}}.
\end{eqnarray*}
Therefore
\begin{eqnarray*}
\mathfrak{M}(t)& =&\exp {\displaystyle{\int_{t_0}^t \left(\frac{r}{2s}+  \frac{\d-K}{ s^\frac{r}{2}}\right)ds}}= C \displaystyle{t^{ \frac{r}{2}} \exp \left(  \frac{2(\d-K)}{2-r }
t^{1-\frac{r}{2}}\right)} .
\end{eqnarray*}
Set
$$m(t):=\displaystyle{t^{ \frac{r}{2}} \exp \left(  \frac{2(\d-K)}{2-r }
t^{1-\frac{r}{2}}\right)}. $$
According to \eqref{valuerate} we have that for some $C_1>0$ 
\begin{eqnarray}\label{valuerate_p}
f(x(t)) - \min_{\cH} f &\leq &
\frac{C_1}{m(t)} \int_{t_0}^t \frac{m(s)}{s^{\frac{3r}{2}}}ds
+\frac{C_1}{m(t)} .
\end{eqnarray}
Note that according to $r<2$, $m(t)$ is an increasing function which has an exponential growth as $t\to +\infty$.  Accordingly, by the mean value theorem we have the following majorization.
\begin{eqnarray}
\frac{1}{m(t)} \int_{t_0}^t \frac{m(s)}{s^{\frac{3r}{2}}}ds
&\leq&  \frac{m(t)}{m(t)} \int_{t_0}^{t} \frac{1}{s^{\frac{3r}{2}}}ds = \mathcal O \left(  \frac{1}{t^{\frac{3r}{2}-1}}    \right).
 \label{valuerate_p2}
\end{eqnarray}

Let us summarize these results in the following statement.

\begin{theorem}\label{RateConvergenceResult-bb}
Take $\e(t)=1/t^r$, $  \frac{2}{3} <r<2$, $\delta >0$.
Let $x : [t_0, +\infty[ \to \mathcal{H}$ be a global trajectory of
$$
\ddot{x}(t) +  \frac{\d}{t^{\frac{r}{2}}}\dot{x}(t) + \nabla f\left(x(t) \right)+ \frac{1}{t^{r}} x(t)=0.
$$
Then, the following convergence rate of the values is satisfied: as $t\to +\infty$
$$
f(x(t))-\min_{\cH} f= \mathcal O \left( \displaystyle{ \frac{1}{t^{\frac{3r}{2}-1}} }   \right).
$$
\end{theorem}

\begin{remark} When $r \to 2$ the exponent $\frac{3r}{2}-1$ tends to $2$.  So there is a continuous transition in the  convergence rate.
As in Remark \ref{Tikh_interest1}
the additional Tikhonov regularization term  is expected to have a regularization effect (even better than in the case $r=2$).
In addition, the above analysis makes appear another critical value, namely $r= \frac{2}{3}$.
\end{remark}

\section{In-depth analysis in the  critical case $\e(t) = c/t^2$}
\label{sec:critical}
Let us refine  our analysis in the case where the Tikhonov regularization coefficient and the damping coefficient are respectively of order $1/t^2$ and $1/t$.
Our analysis will now take into account the coefficients $\alpha$ and $c$ in front of these terms.
So the Cauchy problem for (TRIGS) is written
\begin{align}\label{DynSys-crit}
\begin{cases}
\ddot{x}(t) + \frac{\a}{t} \dot{x}(t) + \nabla f\left(x(t) \right)+\frac{c}{t^2}x(t)=0\\
x(t_0) = x_0, \,
\dot{x}(t_0) = v_0,
\end{cases}
\end{align}
where $t_0 > 0,\,c> 0$, $(x_0,v_0) \in \mathcal{H} \times \mathcal{H},$
and $\a\ge 3$.
%
The starting time $t_0$ is taken  strictly greater than zero to take into account the fact that the functions  $\frac{c}{t^2}$ and $\frac{\a}{t}$  have singularities at $0.$  This is not a limitation of the generality of the proposed approach, since we will focus on the asymptotic behaviour of the generated trajectories.

\subsection{Convergence rate of the values}

\begin{theorem}\label{RateConvergenceResult}
Let $t_0 > 0$ and,  for some initial data $x_0,v_0\in\mathcal{H}$, let $x : [t_0, +\infty[ \to \mathcal{H}$ be the unique global solution of  \eqref{DynSys-crit}.
Then, the following results hold.

$i)$ If $\a=3$, then
$\displaystyle f\left(x(t)\right) - \min_{\cH} f =O\left(\frac{\ln t}{t^2}\right) \mbox{ as }t\to+\infty.$

$ii)$ If $\a>3$, then
$\displaystyle f\left(x(t)\right) - \min_{\cH} f =O\left(\frac{1}{t^2}\right) \mbox{ as }t\to+\infty.$ Further, the trajectory $x$ is bounded and \;
$\displaystyle \|\dot{x}(t)\| =O\left(\frac{1}{t}\right) \mbox{ as }t\to+\infty.$
\end{theorem}

\begin{proof}
The  analysis is parallel to that of Theorem \ref{RateConvergenceResult-a}. Set $f^\ast := f(x^\ast)=\min_{\cH} f$. Let $b:[t_0,+\infty[\to\R$, $b(t)=\frac{K}{t}$ where $K>0$ will be defined later. Let us introduce  $\mathcal{E} : [t_0, +\infty[ \to \mathbb{R},$
\begin{align}\label{Lyapunov-rate}
\mathcal{E}(t) := \left( f(x(t)) - f^\ast \right) +\frac{c}{2t^2}\|x(t)\|^2+ \frac{1}{2} \| b(t)(x(t)-x^\ast) + \dot{x}(t) \|^2,
\end{align}
that will serve as a Lypaunov function. Then,
\begin{align}\label{engderiv1-rate}
\dot{\mathcal{E}}(t) &= \<\n f(x(t)),\dot{x}(t)\>-\frac{c}{t^3}\|x(t)\|^2+\frac{c}{t^2}\<\dot{x}(t),x(t)\>\\
\nonumber&+\<b'(t)(x(t)-x^*)+b(t)\dot{x}(t) +\ddot{x}(t),b(t)(x(t)-x^\ast) +  \dot{x}(t)\>.
\end{align}
According to the dynamic system \eqref{DynSys-crit}, we have
\begin{align}\label{SecondOrderDeriv-rate}
\ddot{x}(t) = - \frac{c}{t^2} x(t)-\frac{\a}{t} \dot{x}(t) - \nabla f(x(t)).
\end{align}
Therefore,
\begin{align}\label{forenergy2-rate}
&\<b'(t)(x(t)-x^*)+b(t)\dot{x}(t) +\ddot{x}(t),b(t)(x(t)-x^\ast) + \dot{x}(t)\>=\\
\nonumber&\left\<-\frac{K}{t^2}(x(t)-x^*)+\frac{K-\a}{t}\dot{x}(t) -\left(\frac{c}{t^2}x(t)+\nabla f(x(t))\right), \frac{K}{t}(x(t)-x^\ast) +  \dot{x}(t)\right\>=\\
\nonumber& -\frac{K^2}{t^3}\|x(t)-x^*\|^2+\frac{K^2-\a K-K}{t^2}\<\dot{x}(t),x(t)-x^*\>+\frac{K-\a}{t}\|\dot{x}(t)\|^2\\
\nonumber&
-\frac{c}{t^2}\<x(t),\dot{x}(t)\> -\<\nabla f(x(t)), \dot{x}(t)\>-\frac{K}{t}\left\<\frac{c}{t^2}x(t)+\nabla f(x(t)), x(t)-x^\ast\right\>.
\end{align}
Combining \eqref{engderiv1-rate} and \eqref{forenergy2-rate}, we get

\begin{align}\label{engderiv2-rate}
\dot{\mathcal{E}}(t) =&-\frac{c}{t^3}\|x(t)\|^2-\frac{K^2}{t^3}\|x(t)-x^*\|^2+\frac{K^2-\a K-K}{t^2}\<\dot{x}(t),x(t)-x^*\>+\frac{K-\a}{t}\|\dot{x}(t)\|^2\\
\nonumber& -\frac{K}{t}\left\<\frac{c}{t^2}x(t)+\nabla f(x(t)), x(t)-x^\ast\right\>.
\end{align}
Consider  the strongly convex function
$$f_t:\mathcal{H}\To\R,\, \; f_t(x)=f(x)+\frac{c}{2t^2}\|x\|^2.$$
From the gradient inequality we have
$$f_t(y)-f_t(x)\ge\<\n f_t(x),y-x\>+\frac{c}{2t^2}\|x-y\|^2,\mbox{ for all }x,y\in\mathcal{H}.$$
Take  $y=x^*$ and $x=x(t)$ in the above inequality.  We obtain
\begin{align*}
&f^*+\frac{c}{2t^2}\|x^*\|^2-f(x(t))-\frac{c}{2t^2}\|x(t)\|^2\ge\\
&-\left\<\n f(x(t))+\frac{c}{t^2}x(t),x(t)-x^*\right\>+\frac{c}{2t^2}\|x(t)-x^*\|^2.
\end{align*}
Consequently,
\begin{align}\label{forenergy4-rate}
-\left\<\frac{c}{t^2}x(t)+\n f(x(t)),x(t)-x^*\right\>\le & -(f(x(t))-f^*)-\frac{c}{2t^2}\|x(t)\|^2-\frac{c}{2t^2}\|x(t)-x^*\|^2\\
\nonumber & +\frac{c}{2t^2}\|x^*\|^2.
\end{align}
By multiplying \eqref{forenergy4-rate} with $\frac{K}{t}$, and injecting in \eqref{engderiv2-rate}, we obtain
\begin{align}\label{engderiv3-rate}
\dot{\mathcal{E}}(t)\le&-\frac{K}{t}(f(x(t))-f^*)-\left(\frac{c}{t^3}+\frac{Kc}{2t^3}\right)\|x(t)\|^2-\left(\frac{K^2}{t^3}+\frac{Kc}{2t^3}\right)\|x(t)-x^*\|^2\\
 \nonumber&+\frac{K^2-\a K-K}{t^2}\<\dot{x}(t),x(t)-x^*\>+\frac{K-\a}{t}\|\dot{x}(t)\|^2+ \frac{cK}{2t^3}\|x^*\|^2.
\end{align}
On the other hand, by multiplying the function $\mathcal{E}(t)$ by $\mu(t)=\frac{\a-K+1}{t}$, we obtain
\begin{align}\label{energ1-rate}
\mu(t)\mathcal{E}(t)=&\frac{\a-K+1}{t} \left( f(x(t)) - f^\ast \right) +\frac{(\a-K+1)c}{2 t^3}\|x(t)\|^2+ \frac{(\a-K+1)K^2}{2t^3} \|x(t)-x^\ast\|^2 \\
\nonumber &+\frac{\a-K+1}{2 t} \| \dot{x}(t) \|^2+ \frac{(\a-K+1)K}{t^2}\<\dot{x}(t),x(t)-x^*\>.
\end{align}
By adding \eqref{engderiv3-rate} and \eqref{energ1-rate}, we get
\begin{align}\label{forgronwal1-rate}
\dot{\mathcal{E}}(t)+\mu(t)\mathcal{E}(t)\le & \frac{\a-2K+1}{t}(f(x(t))-f^*)+\frac{(\a-2K-1)c}{2t^3}\|x(t)\|^2\\
 \nonumber&+\frac{(\a-K-1)K^2-Kc}{2t^3}\|x(t)-x^*\|^2+\frac{K-\a+1}{2t}\|\dot{x}(t)\|^2+ \frac{cK}{2t^3}\|x^*\|^2.
\end{align}

\noindent {\bf The case $\a>3.$}
Take $\frac{\a+1}{2}< K< \a-1.$ Since $\a> 3$, such $K$ exists.
This implies that $\a-2K+1<0$,  hence $\a-2K-1<0$, and $K-\a+1<0$. In addition, since $c>0$ there exists $K\in\left]\frac{\a+1}{2},\a-1\right[$ such that
 \begin{equation}\label{forc}
 (\a-K-1)K^2-Kc\le 0.
 \end{equation}
Indeed, \eqref{forc} can be deduced from the fact that the continuous function $\varphi(K)=(\a-K-1)K$ is decreasing on the interval $\left[\frac{\a+1}{2},\a-1\right]$ and $\varphi \left(\a-1\right)=0$. Therefore, for every $c>0$ there exists $K\in\left]\frac{\a+1}{2},\a-1\right[$ such that $c\ge \varphi(K).$
 So take $K\in\left]\frac{\a+1}{2},\a-1\right[$ such that \eqref{forc} holds. Then, by collecting the previous results, \eqref{forgronwal1-rate} yields
\begin{align}\label{forgronwal11-rate}
&\dot{\mathcal{E}}(t)+\mu(t)\mathcal{E}(t) \le \frac{cK}{2t^3}\|x^*\|^2.
\end{align}
Taking into account that $\mu(t)=\frac{\a-K+1}{t}$, by multiplying \eqref{forgronwal11-rate} with $t^{\a-K+1}$ we get
\begin{align}\label{forgronwal2-rate}
&\frac{d}{dt}\left(t^{\a-K+1}\mathcal{E}(t)\right)\le \frac{cK}{2}\|x^*\|^2 t^{\a-K-2}.
\end{align}
By integrating \eqref{forgronwal2-rate} on $[t_0,t]$, we get
\begin{align}\label{energyrate-rate}
&\mathcal{E}(t)\le \frac{cK\|x^*\|^2}{2(\a-K-1)} \frac{1}{t^2}-\frac{cK\|x^*\|^2}{2(\a-K-1)}\frac{t_0^{\a-K-1}}{t^{\a-K+1}}+\frac{t_0^{\a-K+1}\mathcal{E}(t_0)}{t^{\a-K+1}}.
\end{align}
Since $\a-K+1 >2$, we obtain
\begin{equation}\label{energyrate2-rate}
\mathcal{E}(t)=O\left(\frac{1}{t^2}\right)\mbox{ as }t\to+\infty.
\end{equation}
By definition of $\mathcal{E}(t)$ we immediately deduce that
\begin{equation}\label{decay}
f(x(t))-\min_{\cH} f=O\left(\frac{1}{t^2}\right)\mbox{ as }t\to+\infty,
\end{equation}
and further, that the trajectory $x(\cdot)$ is bounded and
$$\|\dot{x}(t)\|=O\left(\frac{1}{t}\right)\mbox{ as }t\to+\infty.$$

\noindent {\bf The case $\a=3$.}
Take $K=2$. With the previous notations, we have now $\mu(t)=\frac{2}{t}$ and \eqref{forgronwal1-rate} gives
\begin{align}\label{forgronwala3}
\dot{\mathcal{E}}(t)+\frac{2}{t}\mathcal{E}(t)\le & -\frac{c}{t^3}\|x(t)\|^2-\frac{c}{t^3}\|x(t)-x^*\|^2+ \frac{c}{t^3}\|x^*\|^2\le \frac{c}{t^3}\|x^*\|^2.
\end{align}
After multiplication of \eqref{forgronwala3} by $t^2$ we get
\begin{equation}\label{intera3}
\frac{d}{dt}(t^2\mathcal{E}(t))\le \frac{c}{t}\|x^*\|^2.
\end{equation}
By integrating \eqref{intera3} on $[t_0,t]$ we get
\begin{equation}\label{intera31}
\mathcal{E}(t)\le c\|x^*\|^2\frac{\ln t}{t^2}-c\|x^*\|^2\frac{\ln t_0}{t^2}+\frac{t_0^2\mathcal{E}(t_0)}{t^2}.
\end{equation}
Consequently, we have
\begin{equation}\label{energyrate2a3}
\mathcal{E}(t)=O\left(\frac{\ln t}{t^2}\right)\mbox{ as }t\to+\infty.
\end{equation}
By definition of $\mathcal{E}(t)$ we immediately deduce that
\begin{equation}\label{decaya3}
f(x(t))-\min f=O\left(\frac{\ln t}{t^2}\right)\mbox{ as }t\to+\infty.
\end{equation}
which gives the claim.
\end{proof}

\subsection{Strong convergence}

\begin{theorem}\label{StrongConvergence}
Let $t_0 > 0$ and,  for some starting points $x_0,v_0\in\mathcal{H}$, let $x : [t_0, +\infty[ \to \mathcal{H}$ be the unique global solution of  \eqref{DynSys-crit}. Let $x^*$ be the element of minimal norm of $S=\argmin f$, that is
$x^*=\mbox{proj}_{S}0$. Then, for all $\a> 3$ we have that
$$ \liminf_{t \to +\infty}{\| x(t) - x^\ast \|} = 0.$$
Further, if there exists $T \geq t_0$, such that the trajectory $\{ x(t) :t \geq T \}$ stays either in the open ball $B(0, \| x^\ast \|)$ or in its complement, then $x(t)$ converges strongly to $x^*$ as $t\to+\infty.$
\end{theorem}
\begin{proof}
The proof combines energetic and geometric arguments, as it was initiated in \cite{AttCza1}. We successively consider  the three following configurations of the trajectory.

\smallskip

{\bf I.} Assume that  there exists $T\ge t_0$ such that $\|x(t)\|\ge \|x^*\|$ for all $t\ge T.$
Let us denote $f_t(x):=f(x)+\frac{c}{2t^2}\|x\|^2$ and let $x_t :=\argmin f_t(x).$
Let us recall some classical properties of the Tikhonov approximation:
\begin{equation}\label{Tikh-properties}
\forall t>0 \;  \; \|x_{t}\|\le\|x^*\|, \; \mbox{ and }\;
\lim_{t\to+\infty}\| x_t -x^* \|=0.
\end{equation}
\noindent Using the gradient inequality for the strongly convex function $f_t$, we have
$$f_t(x(t))-f_t(x_t)\ge \frac{c}{2t^2}\|x(t)-x_t\|^2.$$
On the other hand
$$f_t(x_t)-f_t(x^*)=f(x_t)-f^*+\frac{c}{2t^2}(\|x_t\|^2-\|x^*\|^2)\ge \frac{c}{2t^2}(\|x_t\|^2-\|x^*\|^2).$$
By adding the last two inequalities we get
\begin{equation}\label{forf}
f_t(x(t))-f_t(x^*)\ge \frac{c}{2t^2}(\|x(t)-x_t\|^2+\|x_t\|^2-\|x^*\|^2),
\end{equation}
Therefore, according to (\ref{Tikh-properties}), to obtain the strong convergence of the trajectory $x(t)$ to $x^*$, it is enough to show that $f_t(x(t))-f_t(x^*)=o\left(\frac{1}{t^2}\right),\mbox{ as }t\to+\infty.$

 For $K>0$, consider now the energy functional
\begin{align}\label{strenergfunc}
E(t)&=f_t(x(t))-f_t(x^*)+\frac{1}{2}\left\|\frac{K}{t}(x(t)-x^*)+\dot{x}(t)\right\|^2\\
\nonumber&=(f(x(t))-f(x^*))+\frac{c}{2t^2}(\|x(t)\|^2-\|x^*\|^2)+\frac{1}{2}\left\|\frac{K}{t}(x(t)-x^*)+\dot{x}(t)\right\|^2.
\end{align}
Then,
\begin{align}\label{strenergy}
\dot{E}(t)=&\<\n f_t(x(t)),\dot{x}(t)\>-\frac{c}{2t^3}(\|x(t)\|^2-\|x^*\|^2)\\
\nonumber&+\left\<-\frac{K}{t^2}(x(t)-x^*)+\frac{K}{t}\dot{x}(t)+\ddot{x}(t),\frac{K}{t}(x(t)-x^*)+\dot{x}(t)\right\>.
\end{align}
Let us examine the different terms of (\ref{strenergy}). According to the constitutive equation \eqref{DynSys-crit} we have
\begin{align}\label{strenergy1}
&\left\<-\frac{K}{t^2}(x(t)-x^*)+\frac{K}{t}\dot{x}(t)+\ddot{x}(t),\frac{K}{t}(x(t)-x^*)+\dot{x}(t)\right\>=\\
\nonumber&\left\<-\frac{K}{t^2}(x(t)-x^*)+\frac{K-\a}{t}\dot{x}(t) -\left(\frac{c}{t^2}x(t)+\nabla f(x(t))\right), \frac{K}{t}(x(t)-x^\ast) +  \dot{x}(t)\right\>=\\
\nonumber& -\frac{K^2}{t^3}\|x(t)-x^*\|^2+\frac{K^2-\a K-K}{t^2}\<\dot{x}(t),x(t)-x^*\>+\frac{K-\a}{t}\|\dot{x}(t)\|^2\\
\nonumber&
-\frac{c}{t^2}\<x(t),\dot{x}(t)\> -\<\nabla f(x(t)), \dot{x}(t)\>-\frac{K}{t}\left\<\frac{c}{t^2}x(t)+\nabla f(x(t)), x(t)-x^\ast\right\>.
\end{align}
Further, from \eqref{forenergy4-rate} we get
\begin{eqnarray}
&&-\frac{K}{t}\left\<\frac{c}{t^2}x(t)+\n f(x(t)),x(t)-x^*\right\>  \nonumber \\
&&\leq  -\frac{K}{t}(f(x(t))-f^*)-\frac{cK}{2t^3}\|x(t)\|^2-\frac{cK}{2t^3}\|x(t)-x^*\|^2 +\frac{cK}{2t^3}\|x^*\|^2 \nonumber  \\
&&=-\frac{K}{t}(f_t(x(t))-f_t(x^*))-\frac{cK}{2t^3}\|x(t)-x^*\|^2 . \label{strenergy2}
\end{eqnarray}
Injecting \eqref{strenergy1} and \eqref{strenergy2}  in \eqref{strenergy} we get
\begin{eqnarray}\label{strenergy4}
\dot{E}(t)&\le&-\frac{K}{t}(f_t(x(t))-f_t(x^*)) -\frac{c}{t^{3}}(\|x(t)\|^2-\|x^*\|^2)-\frac{2K^2+cK}{2t^3}\|x(t)-x^*\|^2 \nonumber  \\
&+&\frac{K^2-\a K-K}{t^2}\<\dot{x}(t),x(t)-x^*\>+\frac{K-\a}{t}\|\dot{x}(t)\|^2. \label{strenergy4b}
\end{eqnarray}
Consider now the function $\mu(t)=\frac{\a+1-K}{t}.$ Then,
\begin{align}\label{forstr4}
\mu(t)E(t)=&\frac{\a+1-K}{t}(f_t(x(t))-f_t(x^*))+\frac{K^2(\a+1-K)}{2t^3}\|x(t)-x^*\|^2\\
\nonumber&+\frac{K(\a+1-K)}{t^2}\<\dot{x}(t),x(t)-x^*\>+\frac{\a+1-K}{2t}\|\dot{x}(t)\|^2.
\end{align}
Consequently, \eqref{strenergy4b} and \eqref{forstr4} yield
\begin{eqnarray}
&&\dot{E}(t)+\mu(t)E(t)\leq \frac{\a+1-2K}{t}(f_t(x(t))-f_t(x^*))-\frac{c}{t^3}(\|x(t)\|^2-\|x^*\|^2) \hspace{1cm}\nonumber \\
 && \hspace{2.5cm}+   \frac{K^2(\a-1-K)-cK}{2t^3}\|x(t)-x^*\|^2+\frac{K-\a+1}{2t}\|\dot{x}(t)\|^2 \nonumber\\
&& = \frac{\a+1-2K}{t}(f(x(t))-f(x^*))+(\a-1-2K)\frac{c}{2t^3}(\|x(t)\|^2-\|x^*\|^2) \nonumber\\
 &&\hspace{2.5cm}+\frac{K^2(\a-1-K)-cK}{2t^3}\|x(t)-x^*\|^2+\frac{K-\a+1}{2t}\|\dot{x}(t)\|^2. \label{strenergy5}
\end{eqnarray}
Assume that  $\frac{\a+1}{2}< K< \a-1.$ Since $\a> 3$ such $K$ exists.
As in the proof of Theorem \ref{RateConvergenceResult} we deduce that $\a-2K+1<0$, $K-\a+1<0$ and since $c>0$ there exists $K\in\left(\frac{\a+1}{2},\a-1\right)$ such that
 \begin{equation}\label{forc1}
 (\a-K-1)K^2-Kc\le 0.
 \end{equation}

 So take $K\in\left(\frac{\a+1}{2},\a-1\right)$ such that \eqref{forc1} holds. 
 Then, \eqref{strenergy5} leads to
\begin{align}\label{strenergy6}
\dot{E}(t)+\frac{\a+1-K}{t}E(t)\le (\a-1-2K)\frac{c}{2t^3}(\|x(t)\|^2-\|x^*\|^2).
\end{align}
Let us integrate the differential inequality \eqref{strenergy6}. After multiplication by $t^{\a+1-K}$ we get
$$\frac{d}{dt}t^{\a+1-K}E(t)\le \frac{c}{2}(\a-1-2K)t^{\a-2-K}(\|x(t)\|^2-\|x^*\|^2)$$
and integrating the latter on $[T,t],\,t>T$ we obtain
\begin{equation}
E(t)\le\frac{c}{2}(\a-1-2K)\frac{\int_{T}^t s^{\a-2-K}(\|x(s)\|^2-\|x^*\|^2)ds}{t^{\a+1-K}}+\frac{T^{\a+1-K}E(T)}{t^{\a+1-K}}.
\end{equation}
In one hand, from the definition of $E(t)$ we have
$$f_t(x(t))-f_t(x^*)\le E(t).$$
Therefore,
$$f_t(x(t))-f_t(x^*)\le \frac{c}{2}(\a-1-2K)\frac{\int_{T}^t s^{\a-2-K}(\|x(s)\|^2-\|x^*\|^2)ds}{t^{\a+1-K}}+\frac{T^{\a+1-K}E(T)}{t^{\a+1-K}}.$$
On the other hand \eqref{forf} gives
$$f_t(x(t))-f_t(x^*)\ge \frac{c}{2t^2}(\|x(t)-x_t\|^2+\|x_t\|^2-\|x^*\|^2).$$
Consequently,
\begin{equation}\label{fontos}
(\a-1-2K)\frac{\int_{T}^t s^{\a-2-K}(\|x(s)\|^2-\|x^*\|^2)ds}{t^{\a-1-K}}+\frac{2T^{\a+1-K}E(T)}{ct^{\a-1-K}}\ge \|x(t)-x_t\|^2+\|x_t\|^2-\|x^*\|^2.
\end{equation}
By assumption  $\|x(t)\|\ge \|x^*\|$ for all $t\ge T$  and $\a-1-2K<0.$
 Hence, for all $t>T$, \eqref{fontos} leads to
\begin{equation}\label{fontoska}
\frac{2 T^{\a+1-K}E(T)}{ct^{\a-1-K}}\ge \|x(t)-x_t\|^2+\|x_t\|^2-\|x^*\|^2.
\end{equation}
Now, by taking the limit $t\To +\infty$ and using that $x_t\to x^*,\,t\to+\infty$ we get
$$\lim_{t\to+\infty}\|x(t)-x_t\|\le 0$$
and hence
$$\lim_{t\to+\infty}x(t)=x^*.$$

{\bf II.} Assume now that there exists $T\ge t_0$ such that $\|x(t)\|< \|x^*\|$ for all $t\ge T.$
According to Theorem \ref{RateConvergenceResult}, we have that
$$
\lim _{t \rightarrow+\infty} f(x(t))=\min_{\cH} f.
$$
Let  $\bar{x} \in \mathcal{H}$ be a weak sequential cluster point of the trajectory $x,$ which exists since, by Theorem \ref{RateConvergenceResult},  the trajectory is bounded. So,  there exists a sequence $\left(t_{n}\right)_{n \in \mathbb{N}} \subseteq[T,+\infty)$ such that $t_{n} \to +\infty$ and $x\left(t_{n}\right)$ converges weakly to $\bar{x}$ as $n \to +\infty$. Since $f$ is weakly lower semicontinuous,  we deduce that
$$
f(\bar{x}) \leq \liminf _{n \rightarrow+\infty} f\left(x\left(t_{n}\right)\right)=\min_{\cH} f \, ,$$ hence $\bar{x} \in \operatorname{argmin} f.$
Now, since the norm is weakly lower semicontinuous, and since $\|x(t)\|< \|x^*\|$ for all $t\ge T$, we have
$$
\begin{array}{c}
\|\bar{x}\| \leq \liminf _{n \rightarrow+\infty}\left\|x\left(t_{n}\right)\right\| \leq\left\|x^\ast \right\|.
\end{array}
$$
Combining $\bar{x} \in \operatorname{argmin} f$ with the definition of $x^\ast$, this implies that $\bar{x}=x^{*}.$ This shows that the trajectory $x(\cdot)$ converges weakly to $x^\ast$. So
$$
\left\|x^\ast \right\| \leq \liminf _{t \rightarrow+\infty}\|x(t)\| \leq \limsup _{t \rightarrow+\infty}\|x(t)\| \leq\left\|x^\ast \right\|,$$
hence we have
$$ \lim _{t \rightarrow+\infty}\|x(t)\|=\left\|x^\ast \right\|.$$
Combining this property with  $x(t) \rightharpoonup x^\ast$ as $t \to +\infty,$ we obtain the strong convergence, that is
$$
\lim _{t \rightarrow+\infty} x(t)=x^\ast.$$

{\bf III.} We suppose that for every $T \geq t_{0}$ there exists $t \geq T$ such that $\left\|x^\ast \right\|>\|x(t)\|$ and also there exists $s \geq T$ such that $\left\|x^{*}\right\| \leq\|x(s)\|$.
From the continuity of $x$, we deduce that there exists a sequence $\left(t_{n}\right)_{n \in \mathbb{N}} \subseteq\left[t_{0},+\infty\right)$ such that $t_{n} \to +\infty$ as $n \to +\infty$ and, for all $n \in \mathbb{N}$ we have
$$
\left\|x\left(t_{n}\right)\right\|=\left\|x^{*}\right\|.$$

 Consider $\bar{x} \in \mathcal{H}$ a weak sequential cluster point of $\left(x\left(t_{n}\right)\right)_{n \in \mathbb{N}}$.
We deduce as in case {\bf II} that  $\bar{x}=x^{*}.$ Hence, $x^*$ is the only weak sequential cluster point of $x(t_n)$ and consequently  the sequence $x(t_n)$ converges weakly to $x^\ast$.
Obviously $\left\|x\left(t_{n}\right)\right\| \to \left\|x^{*}\right\|$ as $n \to +\infty$. So, it follows that $x(t_n)\to x^*,\,n\to+\infty$, that is $\left\|x\left(t_{n}\right)-x^{*}\right\| \to 0$ as $n \to +\infty .$ This leads to
$
\liminf _{t \rightarrow+\infty}\left\|x(t)-x^\ast \right\|=0$.
\end{proof}

\section{Strong convergence-General case}\label{sec:strong}

We are going to analyze via Lyapunov analysis the  strong convergence properties as $t\to +\infty$ of the solution trajectories of the inertial dynamic (TRIGS) that we recall below
\begin{equation*}
 \ddot{x}(t) + \d\sqrt{\e(t)}  \dot{x}(t) + \nabla f (x(t)) + \e (t) x(t) =0.
\end{equation*}
\begin{theorem}\label{StrongConvergenceResult-a}

Let consider the dynamic system {\rm({TRIGS})} where we assume that $\e(\cdot)$ satisfies the condition ${\rm(CD)}_K$ for some $K>0$,
$\int_{t_0}^{+\infty}\e^{\frac32}(t)dt<+\infty$ and
$\lim_{t\to+\infty} \frac{1}{\sqrt{\e(t)}\exp \left({\displaystyle{\int_{t_0}^t (\d-K)\sqrt{\e(s)}ds}}\right)}=0.$

\noindent Then, for any   global solution trajectory $x : [t_0, +\infty[ \to \mathcal{H}$  of  {\rm({TRIGS})},
$$ \liminf_{t \to +\infty}{\| x(t) - x^\ast \|} = 0,$$
where  $x^*$ is the element of minimal norm of $\argmin f$, that is $x^*=\mbox{proj}_{\argmin f}0$.\\
Further, if there exists $T \geq t_0$, such that the trajectory $\{ x(t) :t \geq T \}$ stays either in the open ball $B(0, \| x^\ast \|)$ or in its complement, then $x(t)$ converges strongly to $x^*$ as $t\to+\infty.$
\end{theorem}
\begin{proof}
The proof is parallel to that of Theorem \ref{StrongConvergence}. We analyze the behavior of the trajectory $x(\cdot)$ depending on its position with respect to the ball $B(0, \| x^\ast \|)$.

\smallskip

{\bf I.} Assume that $\|x(t)\|\ge\|x^*\|$ for all $t\ge T.$ Let us denote $f_t(x)=f(x)+\frac{\e(t)}{2}\|x\|^2$, and consider the energy functional $E: \left[t_1,+\infty \right[\to\R$ defined by
$$E(t):=f_t(x(t))-f_t(x^*)+\frac{1}{2}\|c(t)(x(t)-x^*)+\dot{x}(t)\|^2,$$
where $c(t)=K\sqrt{\e(t)}$. 
 Note that
$E(t)=\mathcal{E}(t) -\frac{\e(t)}{2}\|x^*\|^2$,
where $\mathcal{E}(t)$ was defined in the proof of Theorem \ref{RateConvergenceResult-a}. Hence, reasoning as in the proof of  Theorem \ref{RateConvergenceResult-a}, see \eqref{Basic1} (and keeping the term containing $\|x(t)\|^2$ in the right hand side of \eqref{forgronwal1-b}), we get for all $t\ge t_1$ that
\begin{align}\label{strforgronwal1}
\dot{{E}}(t)+\mu(t){E}(t)\le &\left(\frac{\dot{\e}(t)}{2}-c(t)\frac{\e(t)}{2}+\mu(t)\frac{\e(t)}{2}\right)(\|x(t)\|^2-\|x^*\|^2),
\end{align}
where $\mu(t)=-\frac{\dot{\e}(t)}{2\e(t)}+(\d-K)\sqrt{\e(t)}$.
An elementary computation gives  $\frac{\dot{\e}(t)}{2}-c(t)\frac{\e(t)}{2}+\mu(t)\frac{\e(t)}{2} \leq 0$, because of $\e(\cdot)$ decreasing and $K \geq \frac{\delta}{2}$.
Since $\|x(t)\|\ge\|x^*\|$ for all $t\ge T$, \eqref{strforgronwal1} yields
\begin{align}\label{strforgronwal2}
\dot{{E}}(t)+\mu(t){E}(t)\le 0,\mbox{ for all }t\ge T_1=\max\{T,t_1\}.
\end{align}
Set
$$\mathfrak{M}(t)=\exp \left({\displaystyle{\int_{T_1}^t\mu(s)ds}}\right)=\exp \left({\displaystyle{\int_{T_1}^t -\frac{\dot{\e}(s)}{2\e(s)}+(\d-K)\sqrt{\e(s)}ds}}\right).$$
Therefore, we have with $C=\sqrt{\e(T_1)}$
$$\mathfrak{M}(t)=  C\frac{1}{\sqrt{\e(t)}}\exp \left({\displaystyle{\int_{T_1}^t (\d-K)\sqrt{\e(s)}ds}}\right).$$
Multiplying \eqref{strforgronwal2} with $\mathfrak{M}(t)$ and integrating on an interval $[T_1, t],$ we get for all $t\ge T_1$ that
$$\mathfrak{M}(t)E(t)\le \mathfrak{M}(T_1)E(T_1)=C'.$$
Consequently, there exists $C_1'>0$ such that for all $t\ge T_1$ one has
$$E(t)\le \frac{C_1'\sqrt{\e(t)}}{\exp \left({\displaystyle{\int_{T_1}^t (\d-K)\sqrt{\e(s)}ds}}\right)}.$$
Further, $f_t(x(t))-f_t(x^*)\le E(t)$, for all $t\ge t_1$. Therefore,
\begin{equation}\label{forfstr}
f_t(x(t))-f_t(x^*)\le \frac{C_1'\sqrt{\e(t)}}{\exp \left({\displaystyle{\int_{T_1}^t (\d-K)\sqrt{\e(s)}ds}}\right)},\mbox{ for all }t\ge T_1.
\end{equation}
For fixed $t$ let us denote  $x_{\e(t)}=\argmin f_t(x).$  Obviously $\|x_{\e(t)}\|\le\|x^*\|.$\\
Using the gradient inequality for the strongly convex function $f_t$ we have
$$f_t(x)-f_t(x_{\e(t)})\ge \frac{{\e(t)}}{2}\|x-x_{\e(t)}\|^2\mbox{ for all }x\in\cH\mbox{ and }t\ge t_0.$$
On the other hand
$$f_t(x_{\e(t)})-f_t(x^*)=f(x_{\e(t)})-f^*+\frac{{\e(t)}}{2}(\|x_{\e(t)}\|^2-\|x^*\|^2)\ge \frac{{\e(t)}}{2}(\|x_{\e(t)}\|^2-\|x^*\|^2).$$
Now, by adding the last two inequalities we get
\begin{equation}\label{forfgeneral}
f_t(x)-f_t(x^*)\ge \frac{{\e(t)}}{2}(\|x-x_{\e(t)}\|^2+\|x_{\e(t)}\|^2-\|x^*\|^2)\mbox{ for all }x\in\cH\mbox{ and }t\ge t_0.
\end{equation}

Hence, \eqref{forfstr} and \eqref{forfgeneral} lead to
\begin{equation}\label{forfstr1}
\|x(t)-x_{\e(t)}\|^2+\|x_{\e(t)}\|^2-\|x^*\|^2\le \frac{C_2'}{\sqrt{\e(t)}\exp \left({\displaystyle{\int_{T_1}^t (\d-K)\sqrt{\e(s)}ds}}\right)},\mbox{ for all }t\ge T_1.
\end{equation}
Now, by taking the limit as $t\to +\infty$, and using that $x_{\e(t)}\to x^*$ as $t\to+\infty$ and the assumption in the hypotheses of the theorem we get
$\lim_{t\to+\infty}\|x(t)-x_{\e(t)}\|\le 0$,
and hence
$\lim_{t\to+\infty}x(t)=x^*.$

\smallskip

{\bf II.} Assume now, that $\|x(t)\|<\|x^*\|$ for all $t\ge T.$ By Corollary \ref{gen_cor_1}  we get that $f(x(t))\to \min f$ as $t\to+\infty.$
Now, we take $\bar{x} \in \mathcal{H}$ a weak sequential cluster point of the trajectory $x,$ which exists since  the trajectory is bounded. This means that there exists a sequence $\left(t_{n}\right)_{n \in \mathbb{N}} \subseteq[T,+\infty)$ such that $t_{n} \to +\infty$ and $x\left(t_{n}\right)$ converges weakly to $\bar{x}$ as $n \to +\infty$. We know that $f$ is weakly lower semicontinuous, so one has
$$
f(\bar{x}) \leq \liminf _{n \rightarrow+\infty} f\left(x\left(t_{n}\right)\right)=\min f \, ,$$ hence $\bar{x} \in \operatorname{argmin} f.$
Now, since the norm is weakly lower semicontinuous one has that
$$
\begin{array}{c}
\|\bar{x}\| \leq \liminf _{n \rightarrow+\infty}\left\|x\left(t_{n}\right)\right\| \leq\left\|x^\ast \right\|
\end{array}
$$
which, from the definition of $x^\ast$, implies that $\bar{x}=x^{*}.$ This shows that the trajectory $x(\cdot)$ converges weakly to $x^\ast$. So
$$
\left\|x^\ast \right\| \leq \liminf _{t \rightarrow+\infty}\|x(t)\| \leq \limsup _{t \rightarrow+\infty}\|x(t)\| \leq\left\|x^\ast \right\|,$$
hence we have
$$ \lim _{t \rightarrow+\infty}\|x(t)\|=\left\|x^\ast \right\|.$$
From  the previous relation and the fact that $x(t) \rightharpoonup x^\ast$ as $t \to +\infty,$ we obtain the strong convergence, that is
$$
\lim _{t \rightarrow+\infty} x(t)=x^\ast.$$

{\bf III.} We suppose that for every $T \geq t_{0}$ there exists $t \geq T$ such that $\left\|x^\ast \right\|>\|x(t)\|$ and also there exists $s \geq T$ such that $\left\|x^{*}\right\| \leq\|x(s)\|$.
From the continuity of $x$, we deduce that there exists a sequence $\left(t_{n}\right)_{n \in \mathbb{N}} \subseteq\left[t_{0},+\infty\right)$ such that $t_{n} \to +\infty$ as $n \to +\infty$ and, for all $n \in \mathbb{N}$ we have
$$
\left\|x(t_{n})\right\|=\left\|x^{*}\right\|.$$

 Consider $\bar{x} \in \mathcal{H}$ a weak sequential cluster point of $\left(x\left(t_{n}\right)\right)_{n \in \mathbb{N}}$.
We deduce as at case {\bf II} that  $\bar{x}=x^{*}.$ Hence, $x^*$ is the only weak sequential cluster point of $x(t_n)$ and consequently  the sequence $x(t_n)$ converges weakly to $x^\ast$.

Obviously $\left\|x(t_{n})\right\| \to \left\|x^{*}\right\|$ as $n \to +\infty$. So, it follows that $x(t_n)\to x^*,\,n\to+\infty$, that is $\left\|x\left(t_{n}\right)-x^{*}\right\| \to 0$ as $n \to +\infty .$ This leads to
$
\liminf _{t \rightarrow+\infty}\left\|x(t)-x^\ast \right\|=0.$
\end{proof}

\subsection{The case $\e(t)$ is of order $1/t^r$, $\frac{2}{3}< r<2$}
Take $\e(t)=1/t^r$, $\frac{2}{3}<r<2$.
Then,
 $\int_{t_0}^{+\infty}\e^{\frac32}(t)dt= \int_{t_0}^{+\infty}\frac{1}{t^{\frac32 r}}dt<+\infty$,
 $\left(\frac{1}{\sqrt{\e(t)}}\right)'=\frac{r}{2}t^{\frac{r}{2}-1}$ and
$$\lim_{t\to+\infty} \frac{1}{\sqrt{\e(t)}\exp \left({\displaystyle{\int_{t_0}^t (\d-K)\sqrt{\e(s)}ds}}\right)}=
\lim_{t\to+\infty}\frac{C t^\frac{r}{2}}{\exp \left( \frac{2(\d-K)}{2-r}t^{1-\frac{r}{2}}\right)} =0.$$
Therefore, Theorem \ref{StrongConvergenceResult-a} can be applied.
Let us summarize these results in the following statement.

\begin{theorem}\label{StrongRateConvergenceResult-bb}
Take $\e(t)=1/t^r$, $  \frac{2}{3} <r<2$.
Let $x : [t_0, +\infty[ \to \mathcal{H}$ be a global solution trajectory of
$$
\ddot{x}(t) +  \frac{\d}{t^{\frac{r}{2}}}\dot{x}(t) + \nabla f\left(x(t) \right)+ \frac{1}{t^{r}} x(t)=0.
$$
Then, \; \; $ \liminf_{t \to +\infty}{\| x(t) - x^\ast \|} = 0.$

\smallskip

\noindent Further, if there exists $T \geq t_0$, such that the trajectory $\{ x(t) :t \geq T \}$ stays either in the open ball $B(0, \| x^\ast \|)$ or in its complement, then $x(t)$ converges strongly to $x^*$ as $t\to+\infty.$
\end{theorem}

\section{Fast inertial algorithms with Tikhonov regularization}
\label{intro-dyn-RIPA}
On the basis of the convergence properties of continuous dynamic (TRIGS), one would expect to obtain similar results for the algorithms resulting from its temporal discretization.
To illustrate this, we will do a detailed study of the
associated proximal algorithms, obtained by implicit discretization.
A full study of the associated first-order algorithms would be beyond the scope of this article, and will be the subject of further study.
So,  for $k\geq 1$, consider the discrete dynamic
\begin{equation}\label{eq:basic-discret1}
(x_{k+1}-2x_{k}+x_{k-1})+ \frac{\alpha}{k}(x_{k}-x_{k-1})+\nabla f(x_{k+1}) + \frac{c}{k^2} \xi_{k} =0,
\end{equation}
with time step size  equal to one. We take $\xi_k = x_{k}$, which gives
$$ {\rm (IPATRE)}\quad  \left\{
\begin{array}{l}
y_k= x_k+\alpha_{k }(x_k -  x_{k-1})\\
\rule{0pt}{15pt}
x_{k+1}={ \rm prox}_{f}\left( y_k - \frac{c}{k^2}x_k\right),
\end{array}
\right.$$
where (IPATRE) stands for Inertial Proximal Algorithm with Tikhonov REgularization.
 According to \eqref{eq:basic-discret1} we have
\begin{equation}\label{formx}
x_{k+1}=\a_k(x_{k}-x_{k-1})-\nabla f(x_{k+1}) +\left(1- \frac{c}{k^2}\right) x_k.
\end{equation}

\subsection{Convergence of values}
We have the following result.

\begin{theorem}\label{convergencealgorithm}
Let $(x_k)$ be a sequence generated by $\mbox{\rm(IPATRE)}$. Assume that $\a>3$. Then  for all $s\in\left[\frac12,1\right[$ the following  hold:

\smallskip

\begin{itemize}
\item[(i)] $f(x_k)-\min_{\cH} f=o(k^{-2s})$, $\|x_{k}-x_{k-1}\|=o(k^{-s})$ and $\|\n f(x_k)\|=o(k^{-s})$ as $k\to+\infty.$\vspace{2mm}
\item[(ii)] $\ds\sum_{k=1}^{+\infty} k^{2s-1}(f(x_k)-\min_{\cH} f)<+\infty,$ $\ds\sum_{k=1}^{+\infty} k^{2s-1}\|x_k-x_{k-1}\|^2<+\infty$, $\ds\sum_{k=1}^{+\infty} k^{2s}\|\n f(x_k)\|^2<+\infty$.
\end{itemize}
\end{theorem}

\begin{proof} Given $x^*\in\argmin f$, set $f^*=f(x^*)=\min_{\cH} f$.
For $k\ge 2$, consider the discrete energy
\begin{equation}\label{discreteenergy}
E_k :=\|a_{k-1}(x_{k-1}-x^*)+b_{k-1}(x_k-x_{k-1}+\n f(x_k))\|^2+d_{k-1}\|x_{k-1}\|^2,
\end{equation}
where  $a_k=ak^{r-1}, \;  2<a<\a-1$ and $b_k=k^r$, $r\in ]0,1]$. The sequence $(d_k)$ will be defined later. Set shortly  $c_k :=\frac{c}{k^2}.$
Let us develop $E_k$.
\begin{eqnarray}
E_k&=&a_{k-1}^2\|x_{k-1}-x^*\|^2+b_{k-1}^2\|x_k-x_{k-1}\|^2+b_{k-1}^2\|\n f(x_k)\|^2+2a_{k-1}b_{k-1}\<x_k-x_{k-1},x_{k-1}-x^*\> \nonumber\\
&+&2a_{k-1}b_{k-1}\<\n f(x_k),x_{k-1}-x^*\>+2b_{k-1}^2\<\n f(x_k),x_k-x_{k-1}\>+d_{k-1}\|x_{k-1}\|^2.\label{forEk}
\end{eqnarray}
Further
\begin{align*}
&2a_{k-1}b_{k-1}\<x_k-x_{k-1},x_{k-1}-x^*\>=a_{k-1}b_{k-1}(\|x_k-x^*\|^2-\|x_k-x_{k-1}\|^2-\|x_{k-1}-x^*\|^2)\\
&2a_{k-1}b_{k-1}\<\n f(x_k),x_{k-1}-x^*\>=2a_{k-1}b_{k-1}\<\n f(x_k),x_{k}-x^*\>-2a_{k-1}b_{k-1}\<\n f(x_k),x_{k}-x_{k-1}\>.
\end{align*}
Consequently, \eqref{forEk} becomes
\begin{eqnarray}
E_k=a_{k-1}b_{k-1}\|x_k-x^*\|^2+(a_{k-1}^2-a_{k-1}b_{k-1})\|x_{k-1}-x^*\|^2+(b_{k-1}^2-a_{k-1}b_{k-1})\|x_k-x_{k-1}\|^2 \nonumber \\
+b_{k-1}^2\|\n f(x_k)\|^2+2a_{k-1}b_{k-1}\<\n f(x_k),x_{k}-x^*\>+(2b_{k-1}^2-2a_{k-1}b_{k-1})\<\n f(x_k),x_k-x_{k-1}\> \nonumber\\
  +d_{k-1}\|x_{k-1}\|^2 . \hspace{8cm}\label{forEk1}
\end{eqnarray}
Let us proceed similarly with $E_{k+1}$. Let us first observe  that from (\ref{discreteenergy}) we have
$$E_{k+1}=\|a_{k}(x_{k}-x^*)+b_{k}(\a_k(x_k-x_{k-1})-c_k x_k)\|^2+d_{k}\|x_{k}\|^2.$$
Therefore, after development we get
\begin{eqnarray}
E_{k+1}&=&a_k^2\|x_k-x^*\|^2+\a_k^2b_k^2\|x_k-x_{k-1}\|^2+b_k^2c_k^2\|x_k\|^2+2\a_k a_kb_k\<x_k-x_{k-1},x_k-x^*\> \nonumber \\
&&-2\a_kb_k^2c_k\<x_k-x_{k-1},x_k\>-2a_kb_kc_k\<x_k,x_k-x^*\>+d_k\|x_k\|^2.\label{forEk+1}
\end{eqnarray}
Further,
\begin{eqnarray*}
&&2\a_k a_kb_k\<x_k-x_{k-1},x_k-x^*\>=-\a_k a_kb_k(\|x_{k-1}-x^*\|-\|x_k-x_{k-1}\|^2-\|x_k-x^*\|^2) \\
&&-2\a_kb_k^2c_k\<x_k-x_{k-1},x_k\>=\a_kb_k^2c_k(\|x_{k-1}\|^2-\|x_k-x_{k-1}\|^2-\|x_k\|^2)
\\
&&-2a_kb_kc_k\<x_k,x_k-x^*\>=a_kb_kc_k(\|x^*\|^2-\|x_k-x^*\|^2-\|x_k\|^2).
\end{eqnarray*}
Therefore, (\ref{forEk+1}) yields
\begin{align}\label{forEk+11}
E_{k+1}&=(a_k^2+\a_k a_kb_k-a_kb_kc_k)\|x_k-x^*\|^2-\a_k a_kb_k\|x_{k-1}-x^*\|^2
\\
\nonumber& +(\a_k^2b_k^2+\a_k a_kb_k-\a_kb_k^2c_k)\|x_k-x_{k-1}\|^2+(b_k^2c_k^2+d_k-\a_kb_k^2c_k -a_kb_kc_k)\|x_k\|^2\\
\nonumber&+\a_kb_k^2c_k\|x_{k-1}\|^2+a_kb_kc_k\|x^*\|^2.
\end{align}
By combining \eqref{forEk1} and  \eqref{forEk+11}, we obtain
\begin{align}
&E_{k+1}-E_k=(a_k^2+\a_k a_kb_k-a_kb_kc_k-a_{k-1}b_{k-1})\|x_k-x^*\|^2 \nonumber\\
\nonumber&+(-\a_k a_kb_k -a_{k-1}^2+a_{k-1}b_{k-1})\|x_{k-1}-x^*\|^2\\
\nonumber& +(\a_k^2b_k^2+\a_k a_kb_k-\a_kb_k^2c_k-b_{k-1}^2+a_{k-1}b_{k-1})\|x_k-x_{k-1}\|^2\\
\nonumber&+(b_k^2c_k^2+d_k-\a_kb_k^2c_k -a_kb_kc_k)\|x_k\|^2+(\a_kb_k^2c_k-d_{k-1})\|x_{k-1}\|^2-b_{k-1}^2\|\n f(x_k)\|^2\\
&+2a_{k-1}b_{k-1}\<\n f(x_k),x^*-x_{k}\>+(2b_{k-1}^2-2a_{k-1}b_{k-1})\<\n f(x_k),x_{k-1}-x_k\>
+a_kb_kc_k\|x^*\|^2. \label{energdif}
\end{align}
By convexity of $f$, we have
$$\<\n f(x_k),x^*-x_{k}\>\le f^*-f(x_k)\mbox{ and }\<\n f(x_k),x_{k-1}-x_k\>\le f(x_{k-1})-f(x_k).$$
According to the form of $(a_k)$ and $(b_k)$,  there exists  $k_0\ge 2$ such that  $b_k\ge a_k$ for all $k \geq k_0$.
Consequently, $2b_{k-1}^2-2a_{k-1}b_{k-1} \geq0$ which, according to the above convexity inequalities, gives
\begin{eqnarray}
&& \; \;2a_{k-1}b_{k-1}\<\n f(x_k),x^*-x_{k}\>+(2b_{k-1}^2-2a_{k-1}b_{k-1})\<\n f(x_k),x_{k-1}-x_k\> \label{forfrate}\\
&& \leq 2a_{k-1}b_{k-1}(f^* -f(x_k)) +(2b_{k-1}^2-2a_{k-1}b_{k-1})\left[ f(x_{k-1})-f(x_k)  \right] \nonumber\\
&&= -2a_{k-1}b_{k-1}(f(x_k)-f^*) +(2b_{k-1}^2-2a_{k-1}b_{k-1})\left[ (f(x_{k-1})- f^* )- (f(x_k) - f^* )   \right] \nonumber\\
&&= (2b_{k-1}^2-2a_{k-1}b_{k-1})(f(x_{k-1})-f^*) -2b_{k-1}^2 (f(x_k)-f^*) \nonumber\\
&&=(2b_{k-1}^2-2a_{k-1}b_{k-1})(f(x_{k-1})-f^*)-\Big((2b_{k}^2-2a_{k}b_{k})+(2b_{k-1}^2-2b_{k}^2+2a_{k}b_{k})\Big)(f(x_{k})-f^*).\nonumber
\end{eqnarray}
Set $\mu_k:=2b_{k}^2-2a_{k}b_{k}$ and observe that $\mu_k\ge 0$ for all $k\ge k_0$, and $\mu_k \sim C k^{2r}$ (we use $C$ as a generic positive constant).
Let us also introduce $m_k:=2b_{k-1}^2-2b_{k}^2+2a_{k}b_{k}$, and observe that
  $m_k\ge 0$ for all $k\ge k_0$.
Equivalently, let us show that for all $\frac12\le r\le 1$ one has $b_{k}^2-a_{k}b_{k}\le b_{k-1}^2$ for all $k\ge 1$. Equivalently
$ k^{2r}-ak^{2r-1}-(k-1)^{2r} \leq 0 $.
 By convexity of the function
$x\mapsto x^{2r}$, the subgradient inequality gives
$$
(x-1)^{2r} \geq x^{2r} -2r x^{2r -1} \geq x^{2r} -a x^{2r -1},
$$
where the second inequality comes from $  2 r < a $.
Replacing $x$ with $k$ gives the claim.
In addition $m_k \sim Ck^{2r-1}.$\;
Combining \eqref{energdif} and \eqref{forfrate}, we obtain that for all $k\ge k_0$
\begin{align}\label{forfrate1}
&E_{k+1}-E_k+\mu_k(f(x_{k})-f^*)-\mu_{k-1}(f(x_{k-1})-f^*)+m_k(f(x_{k})-f^*)\\
\nonumber& \le(a_k^2+\a_k a_kb_k-a_kb_kc_k-a_{k-1}b_{k-1})\|x_k-x^*\|^2\\
\nonumber&+(-\a_k a_kb_k -a_{k-1}^2+a_{k-1}b_{k-1})\|x_{k-1}-x^*\|^2\\
\nonumber& +(\a_k^2b_k^2+\a_k a_kb_k-\a_kb_k^2c_k-b_{k-1}^2+a_{k-1}b_{k-1})\|x_k-x_{k-1}\|^2\\
\nonumber&+(b_k^2c_k^2+d_k-\a_kb_k^2c_k -a_kb_kc_k)\|x_k\|^2+(\a_kb_k^2c_k-d_{k-1})\|x_{k-1}\|^2-b_{k-1}^2\|\n f(x_k)\|^2\\
\nonumber&+a_kb_kc_k\|x^*\|^2.
\end{align}
Let us now analyze the right hand side of \eqref{forfrate1}.

\smallskip

$i) $ Write the coefficient of  $\|x_k-x^*\|^2$ so as to show a  term similar to  the coefficient of  $\|x_{k-1}-x^*\|^2$. This will prepare the summation of these quantities. This gives
\begin{eqnarray}
&& a_k^2+\a_k a_kb_k-a_kb_kc_k-a_{k-1}b_{k-1}=(\a_{k+1}a_{k+1}b_{k+1} +a_{k}^2-a_{k}b_{k}) \label{anchor1}\\
\nonumber&&+(\a_k a_kb_k-a_kb_kc_k-a_{k-1}b_{k-1}-\a_{k+1}a_{k+1}b_{k+1}+a_{k}b_{k}).
\end{eqnarray}
a) \; By definition, $\a_{k+1}a_{k+1}b_{k+1} +a_{k}^2-a_{k}b_{k}=a(k+1)^{2r-1}-\a a(k+1)^{2r-2}+a^2k^{2r-2}-a k^{2r-1}$. Proceeding as before, let us show that $a(x+1)^{2r-1}-\a a(x+1)^{2r-2}+a^2x^{2r-2}-a x^{2r-1}\le 0$ for $x$ large enough.  By taking $ \demi \leq r \leq 1 $, by convexity of the function
$x\mapsto -x^{2r-1}$, the subgradient inequality gives
$
(2r-1)x^{2r-2} \geq (x+1)^{2r-1} - x^{2r -1}.
$
Therefore,
$$a(x+1)^{2r-1}-a x^{2r-1}-\a a(x+1)^{2r-2}+a^2x^{2r-2}\le a(2r-1)x^{2r-2}-\a a(x+1)^{2r-2}+a^2x^{2r-2}.$$
But $a(2r-1)x^{2r-2}+a^2x^{2r-2}\le \a a(x+1)^{2r-2}$ since $2r+a\le\a+1$ and the claim follows.\\
Therefore, there exists $k_1\ge k_0$ such that for all $\frac12\le r\le 1$ we have
\begin{equation}\label{signxkx}
\a_{k+1}a_{k+1}b_{k+1} +a_{k}^2-a_{k}b_{k}\le 0,\mbox{ for all }k\ge k_1.
\end{equation}
Set $\nu_k:=-\a_{k+1}a_{k+1}b_{k+1} -a_{k}^2+a_{k}b_{k}$.   According to (\ref{signxkx}), $\nu_k\ge 0$ for all $k\ge k_1$, and   $\nu_k \sim Ck^{2r-2}$.

\smallskip

\noindent b) \; Consider now the second term in the right hand side of (\ref{anchor1}):
\begin{eqnarray*}
&&a_k a_kb_k-a_kb_kc_k-a_{k-1}b_{k-1}-\a_{k+1}a_{k+1}b_{k+1}+a_{k}b_{k}\\
&&=2ak^{2r-1}-\a ak^{2r-2}-ack^{2r-3}-a(k-1)^{2r-1}-
a(k+1)^{2r-1}+\a a(k+1)^{2r-2}.
\end{eqnarray*}
 Let us show that  for all $\frac12\le r\le 1$
$$\phi(x,r)=2ax^{2r-1}-\a ax^{2r-2}-acx^{2r-3}-a(x-1)^{2r-1}-a(x+1)^{2r-1}+\a a(x+1)^{2r-2}\le 0$$
for $x$ large enough.
By convexity of the function
$x\mapsto x^{2r-1}-(x-1)^{2r-1}$ (one can easily verify that its  second order derivative is nonnegative), the subgradient inequality gives
$(x+1)^{2r-1}-2x^{2r-1}+(x-1)^{2r-1}\ge (2r-1)(x^{2r-2}-(x-1)^{2r-2})$.
Therefore
\begin{eqnarray*}
\phi(x,r)&=& -a [ (x+1)^{2r-1}-2x^{2r-1}+(x-1)^{2r-1}  ]            -\a ax^{2r-2}-acx^{2r-3}+ \a a(k+1)^{2r-2}\\
&\le& -a [ (2r-1)(x^{2r-2}-(x-1)^{2r-2})  ]-\a ax^{2r-2}-acx^{2r-3}+ \a a(k+1)^{2r-2}\\
&=& a (2r-1)(x-1)^{2r-2}-a(\a+2r-1)x^{2r-2}-acx^{2r-3}+\a a(x+1)^{2r-2}.
\end{eqnarray*}
Similarly, by convexity of the function
$x\mapsto (x-1)^{2r-2}-x^{2r-2}$, the subgradient inequality gives
$2x^{2r-2}-(x+1)^{2r-2}-(x-1)^{2r-2}\ge (2r-2)((x-1)^{2r-3}-x^{2r-3})$.
Therefore, $a\a (x+1)^{2r-2}-a\a x^{2r-2}\le a\a (x^{2r-2}-(x-1)^{2r-2})-a\a (2r-2)((x-1)^{2r-3}-x^{2r-3}).$ Consequently,
$$\phi(x,r)\le  a (2r-1-\a)((x-1)^{2r-2}-x^{2r-2})-a\a (2r-2)((x-1)^{2r-3}-x^{2r-3})-acx^{2r-3}.$$
Finally, by convexity of the function
$x\mapsto x^{2r-2}$, the subgradient inequality gives $(x-1)^{2r-2}-x^{2r-2}\ge -(2r-2)x^{2r-3}$.
Taking into account that $a (2r-1-\a)\le 0$ we get
$$\phi(x,r)\le -a (2r-1-2\a)(2r-2)x^{2r-3}-a\a (2r-2)(x-1)^{2r-3}-acx^{2r-3}.$$
Since $\frac{2\a+1-2r}{\a}>1$ we obtain that $\phi(x,r)\le 0$ for $x>1$.\\
Consequently, there exists  $k_2\ge k_1$ such that for all $\frac12\le r\le 1$
\begin{equation}\label{forstrxk}
\a_k a_kb_k-a_kb_kc_k-a_{k-1}b_{k-1}-\a_{k+1}a_{k+1}b_{k+1}+a_{k}b_{k}\le 0,\mbox{ for all }k\ge k_2.
\end{equation}
Set $n_k:=-\a_k a_kb_k+a_kb_kc_k+a_{k-1}b_{k-1}+\a_{k+1}a_{k+1}b_{k+1}-a_{k}b_{k}$. So  $n_k\ge 0$ for all $k\ge k_2$ and $n_k\sim Ck^{2r-3}.$

\smallskip

$ii)$ \; Let us now examine the coefficient of $\|x_k-x_{k-1}\|^2$. By definition we have
\begin{eqnarray*}
&&\a_k^2b_k^2+\a_k a_kb_k-\a_kb_k^2c_k-b_{k-1}^2+a_{k-1}b_{k-1}\\
&&= k^{2r}-(k-1)^{2r}+(-2\a +a)k^{2r-1}+a(k-1)^{2r-1}+(\a^2 -\a a-c)k^{2r-2}+\a  ck^{2r-3}.
\end{eqnarray*}
Let us show that  for all $\frac12\le r\le 1$
$$\phi(x,r)= x^{2r}-(x-1)^{2r}+(-2\a +a)x^{2r-1}+a(x-1)^{2r-1}+(\a^2 -\a a-c)x^{2r-2}+\a  cx^{2r-3}\le 0,$$
if $x$ is large enough.
By convexity of the function
$x\mapsto x^{2r}-a x^{2r-1}$, the subgradient  inequality gives
$((x-1)^{2r}-a(x-1)^{2r-1})-(x^{2r}-ax^{2r-1})\ge -(2rx^{2r-1}-a(2r-1)x^{2r-2})$.
Therefore, taking into account that $r-\a+a\le1-\a+a\le 0$, we obtain
$$\phi(x,r)\le 2(r-\a+a)x^{2r-1}-a(2r-1)x^{2r-2}+(\a^2 -\a a-c)x^{2r-2}+\a  cx^{2r-3}\le 0,$$
for $x$ large enough.
Consequently, there exist $k_3\ge k_2$ such that for all $\frac12\le r\le 1$
\begin{equation}\label{forspeed}
\a_k^2b_k^2+\a_k a_kb_k-\a_kb_k^2c_k-b_{k-1}^2+a_{k-1}b_{k-1}\le0,\mbox{ for all }k\ge k_3.
\end{equation}
Set $\eta_k:=-\a_k^2b_k^2-\a_k a_kb_k+\a_kb_k^2c_k+b_{k-1}^2-a_{k-1}b_{k-1}$. So  $\eta_k\ge 0$ for all $k\ge k_3$ and $\eta_k\sim Ck^{2r-1}.$

\smallskip

$iii)$ \;\; The coefficient of $\|x_{k-1}\|^2$ is $\a_kb_k^2c_k-d_{k-1}$. We proceed in a similar way as in $i)$,  and write the coefficient of $\|x_k\|^2$ as
$$b_k^2c_k^2+d_k-\a_kb_k^2c_k -a_kb_kc_k=(-\a_{k+1}b_{k+1}^2c_{k+1}+d_{k})+(b_k^2c_k^2+\a_{k+1}b_{k+1}^2c_{k+1}-\a_kb_k^2c_k -a_kb_kc_k).$$
We have
\begin{align*}
b_k^2c_k^2+\a_{k+1}b_{k+1}^2c_{k+1}-\a_kb_k^2c_k -a_kb_kc_k&=c^2k^{2r-4}+ c(k+1)^{2r-2}-\a c(k+1)^{2r-3}\\
&-ck^{2r-2}+\a ck^{2r-3}-ack^{2r-3}.
\end{align*}
Let us show that  for all $\frac12\le r\le 1$
$$\phi(x,r)=c(x+1)^{2r-2}-\a c(x+1)^{2r-3}-cx^{2r-2}+\a cx^{2r-3}-acx^{2r-3}+c^2x^{2r-4}\le 0$$
for $x$ large enough.
Since for $x$ large enough, the function $x\mapsto x^{2r-2}-\a x^{2r-3}$ is convex,  the subgradient inequality gives
$$ (x^{2r-2}-\a x^{2r-3})-( (x+1)^{2r-2}-\a (x+1)^{2r-3})\ge -((2r-2)(x+1)^{2r-3}-\a (2r-3)(x+1)^{2r-4}).$$
Therefore, by taking into account that $r\le 1$, we obtain
$$\phi(x,r)\le (2r-2)c(x+1)^{2r-3}-\a (2r-3)c(x+1)^{2r-4}-acx^{2r-3}+c^2x^{2r-4}\le 0$$
for $x$ large enough.
Consequently,  there exists $k_4\ge k_3$ such that for all $\frac12\le r\le 1$ we have
\begin{equation}\label{signxk}
b_k^2c_k^2+\a_{k+1}b_{k+1}^2c_{k+1}-\a_kb_k^2c_k -a_kb_kc_k\le 0\mbox{ for all }k\ge k_4.
\end{equation}
Let us denote $\sigma_k:=\a_{k+1}b_{k+1}^2c_{k+1}-d_k$ and $s_k:=-b_k^2c_k^2-\a_{k+1}b_{k+1}^2c_{k+1}+\a_kb_k^2c_k +a_kb_kc_k$ and observe that $s_k\ge 0$ for all $k\ge k_4$ and $s_k\sim Ck^{2r-3}.$\\
Combining \eqref{forfrate1}, \eqref{signxkx}, \eqref{forstrxk}, \eqref{forspeed} and \eqref{signxk} we obtain that
for all $k\ge k_4$ and $r\in\left[\frac12,1\right]$ it holds
\begin{align}\label{forfrate2}
&E_{k+1}-E_k+\mu_k(f(x_{k})-f^*)-\mu_{k-1}(f(x_{k-1})-f^*)+m_k(f(x_{k})-f^*)\\
\nonumber&+\nu_k\|x_k-x^*\|^2-\nu_{k-1}\|x_{k-1}-x^*\|^2+n_k\|x_k-x^*\|^2\\
\nonumber&+\sigma_k\|x_k\|^2-\sigma_{k-1}\|x_{k-1}\|^2+s_k\|x_k\|^2\\
\nonumber& +\eta_k\|x_k-x_{k-1}\|^2+b_{k-1}^2\|\n f(x_k)\|^2\le a_kb_kc_k\|x^*\|^2.
\end{align}
Finally, take $d_{k-1}=\frac12 \a_kb_k^2c_k.$ Then, $\sigma_k=\frac12\a_{k+1}b_{k+1}^2c_{k+1}\sim C k^{2r-2}$, $\sigma_k\ge 0$ for all $k\ge k_5=\max(\a-1, k_4).$ Further, $\mu_k, m_k, \nu_k, n_k, s_k$ and $ \eta_k$ are nonnegative for all $k\ge k_5$ and $r\in\left[\frac12,1\right]$.

\smallskip

Assume now that $\frac12\le r<1.$
According to $\sum_{k\ge k_5}a_kb_kc_k\|x^*\|^2=ac\|x^*\|^2\sum_{k\ge k_5}k^{2r-3}=C<+\infty$, by summing up \eqref{forfrate2} from $k=k_5$ to $k=n>k_5$, we obtain that there exists $C_1>0$ such that
\begin{eqnarray*}
&&E_{n+1}\le C_1,\\
&&\mu_n(f(x_{n})-f^*)\le C_1,\mbox{ hence }f(x_{n})-f^*=\mathcal{O}(n^{-2r}),\\
&&\sum_{k\ge k_5}m_k(f(x_{k})-f^*)\le C_1,\mbox{ hence }\sum_{k\ge 1}k^{2r-1}(f(x_{k})-f^*)<+\infty,\\
&&\nu_k\|x_k-x^*\|^2\le C_1,\mbox{ hence }\|x_n-x^*\|=\mathcal{O}(n^{1-r}),\\
&&\sum_{k\ge k_5} n_k\|x_k-x^*\|^2\le C_1,\mbox{ hence }\sum_{k\ge 1} k^{2r-3}\|x_k-x^*\|^2<+\infty,\\
&&\sigma_k\|x_k\|^2\le C_1,\mbox{ hence }\|x_n\|=\mathcal{O}(n^{1-r}),\\
&&\sum_{k\ge k_5} s_k\|x_k\|^2\le C_1,\mbox{ hence }\sum_{k\ge 1} k^{2r-3}\|x_k\|^2<+\infty,\\
&&\sum_{k\ge k_5} \eta_k\|x_k-x_{k-1}\|^2\le C_1,\mbox{ hence }\sum_{k\ge 1}k^{2r-1}\|x_k-x_{k-1}\|^2<+\infty\\
&&\sum_{k\ge k_5} b_{k-1}^2\|\n f(x_k)\|^2\le C_1,\mbox{ hence }\sum_{k\ge 1}k^{2r}\|\n f(x_k)\|^2<+\infty.
\end{eqnarray*}
Since $\sum_{k\ge 1}k^{2r}\|\n f(x_k)\|^2<+\infty$, we have $\|\n f(x_n)\|=o(n^{-r}).$ Combining this property with $E_{n+1}\le C_1$ yields $\sup_{n\ge 1}\|an^{r-1}(x_n-x^*)+n^{r}(x_{n+1}-x_{n})\|+\frac{c}{2}\left(1-\frac{\a}{n}\right)n^{2r-2}\|x_{n-1}\|^2<+\infty$.\\
Let us show now, that $f(x_n)-f^*= o(n^{-2r})$ and $\|x_n-x_{n-1}\|=o(n^{-r}).$
From \eqref{forfrate2} we get
$$\sum_{k\ge 1} [(E_{k+1}+\mu_k (f(x_{k})-f^*)+\nu_k\|x_k-x^*\|^2)-(E_{k}+\mu_{k-1} (f(x_{k-1})-f^*)+\nu_{k-1}\|x_{k-1}-x^*\|^2)]_+<+\infty.$$
Therefore, the following  limit exists
$$\lim_{k\to+\infty}(\|ak^{r-1}(x_k-x^*)+k^{r}(x_{k+1}-x_{k})\|^2+d_k\|x_k\|^2+\mu_k (f(x_{k})-f^*)+\nu_k\|x_k-x^*\|^2).$$
Note that $d_k \sim C k^{2r-2},\,\mu_k\sim C k^{2r}$ and $\nu_k\sim C k^{2r-2}.$ Further, we have\\
$\sum_{k\ge 1} k^{2r-3}\|x_k-x^*\|^2<+\infty$, $\sum_{k\ge 1} k^{2r-1}\|x_k-x_{k-1}\|^2<+\infty$, $\sum_{k\ge 1}k^{2r-1}(f(x_{k})-f^*)< +\infty$ and $\sum_{k\ge 1} k^{2r-3}\|x_k\|^2<+\infty$, hence
$$\sum_{k\ge 1}\frac{1}{k}(\|ak^{r-1}(x_k-x^*)+k^{r}(x_{k+1}-x_{k})\|^2+d_k\|x_k\|^2+\mu_k (f(x_{k})-f^*)+\nu_k\|x_k-x^*\|^2)<+\infty.$$
Since $\sum_{k\ge 1}\frac{1}{k}=+\infty$ we get
$$\lim_{k\to+\infty}(\|ak^{r-1}(x_k-x^*)+k^{r}(x_{k+1}-x_{k})\|^2+d_k\|x_k\|^2+\mu_k (f(x_{k})-f^*)+\nu_k\|x_k-x^*\|^2)=0$$
and the claim follows.
\end{proof}
\begin{remark}
The convergence rate of the values is $f(x_k)-\min_{\cH} f=o(k^{-2s})$ for any $0 <s <1$.
Practically it is as good as the rate $ f\left(x(t)\right) - \min_{\cH} f =O\left(\frac{1}{t^2}\right) $ obtained for the continuous dynamic.
\end{remark}

\subsection{Strong convergence to the minimum norm solution}
\begin{theorem}\label{strconvergencealgorithm}
Take $\a>3$. Let $(x_k)$ be a sequence generated by {\rm(IPATRE)}. Let $x^*$ be the minimum norm element of $\argmin f$. Then, $\liminf_{k\to+\infty}\|x_k-x^*\|=0$. Further, $(x_k)$ converges strongly to $x^*$ whenever $(x_k)$ is in the interior  of the ball $B(0,\|x^*\|)$ for $k$ large enough,  or $(x_k)$ is in the complement of the ball $B(0,\|x^*\|)$ for $k$ large enough.
\end{theorem}
\begin{proof}
{\bf Case I.}
Assume that  there exists $k_0\in\N$ such that $\|x_k\|\ge \|x^*\|$ for all $k\ge k_0.$
Set $c_k=\frac{c}{k^2},$ and define $f_{c_k}(x):=f(x)+\frac{c}{2k^2}\|x\|^2$.
Consider  the energy function defined in \eqref{discreteenergy} with $r=1$, that is $a_k=a$ and $b_k=k^2$, where we assume that $\max(2,\a-2)<a<\a-1.$ Then,
$$E_k=\|a(x_{k-1}-x^*)+(k-1)^2(x_k-x_{k-1}+\n f(x_k))\|^2+d_{k-1}\|x_{k-1}\|^2,$$
where the sequence  $(d_k)$ will be defined later. Next, we introduce another energy functional
 \begin{equation}
\mathcal{E}_k =\frac12 c_{k-1}(\|x_{k-1}\|^2-\|x^*\|^2)+\|a(x_{k-1}-x^*)+(k-1)^2(x_k-x_{k-1}+\nabla f(x_k))\|^2 + d_{k-1}\|x_{k-1}\|^2. \label{discstrenergfunc}
\end{equation}
Note that $\mathcal{E}_k=\frac12c_{k-1}(\|x_{k-1}\|^2-\|x^*\|^2)+E_k$.
Then,
\begin{align}\label{discstrenergfunc1}
\mathcal{E}_{k+1}-\mathcal{E}_k=\frac12c_{k}(\|x_{k}\|^2-\|x^*\|^2)-\frac12c_{k-1}(\|x_{k-1}\|^2-\|x^*\|^2)+E_{k+1}-E_k.
\end{align}
%
According to \eqref{forfrate2}, there exists $k_1\ge k_0$ such that for all $k\ge k_1$
\begin{align}\label{forfratestr}
&\mathcal{E}_{k+1}-\mathcal{E}_k+\mu_k(f(x_{k})-f^*)-\mu_{k-1}(f(x_{k-1})-f^*)+m_k(f(x_{k})-f^*)\\
\nonumber&+\nu_k\|x_k-x^*\|^2-\nu_{k-1}\|x_{k-1}-x^*\|^2+n_k\|x_k-x^*\|^2\\
\nonumber& +\eta_k\|x_k-x_{k-1}\|^2+b_{k-1}^2\|\n f(x_k)\|^2\leq
-\sigma_k\|x_k\|^2+\sigma_{k-1}\|x_{k-1}\|^2-s_k\|x_k\|^2\\
\nonumber& +\frac12c_{k}(\|x_{k}\|^2-\|x^*\|^2)-\frac12c_{k-1}(\|x_{k-1}\|^2-\|x^*\|^2)+ a_kb_kc_k\|x^*\|^2.
\end{align}
Adding $\frac12(\mu_k+m_k) c_k(\|x_k\|^2-\|x^*\|^2)-\frac12\mu_{k-1} c_{k-1}(\|x_{k-1}\|^2-\|x^*\|^2)$ to both side of \eqref{forfratestr}
we get
\begin{eqnarray}
&&\mathcal{E}_{k+1}-\mathcal{E}_k+\mu_k(f_{c_k}(x_{k})-f_{c_k}(x^*))-\mu_{k-1}(f_{c_{k-1}}(x_{k-1})-f_{c_{k-1}}(x^*))+m_k(f_{c_k}(x_{k})-f_{c_k}(x^*))\nonumber \\
&&+\nu_k\|x_k-x^*\|^2-\nu_{k-1}\|x_{k-1}-x^*\|^2+n_k\|x_k-x^*\|^2  +\eta_k\|x_k-x_{k-1}\|^2 +b_{k-1}^2\|\n f(x_k)\|^2 \label{forfratestr1}\\
\nonumber&&  \leq
-\sigma_k\|x_k\|^2+\sigma_{k-1}\|x_{k-1}\|^2-s_k\|x_k\|^2\\
\nonumber&& +\frac12(\mu_k+m_k+1)c_{k}(\|x_{k}\|^2-\|x^*\|^2)-\frac12(\mu_{k-1}+1)c_{k-1}(\|x_{k-1}\|^2-\|x^*\|^2)+ a_kb_kc_k\|x^*\|^2.
\end{eqnarray}
The right hand side of \eqref{forfratestr1} can be written as
\begin{align*}
&\left(\frac12(\mu_k+m_k+1)c_{k}-\sigma_k-s_k\right)(\|x_{k}\|^2-\|x^*\|^2)\\
&+\left(-\frac12(\mu_{k-1}+1)c_{k-1}+\sigma_{k-1}\right)(\|x_{k-1}\|^2-\|x^*\|^2)
+(a_kb_kc_k-\sigma_k-s_k+\sigma_{k-1})\|x^*\|^2.
\end{align*}
In this case we have  $\mu_k=2b_{k}^2-2a_{k}b_{k}=2k^2-2ak$  and $m_k=2b_{k-1}^2-2b_{k}^2+2a_{k}b_{k}=2(a-2)k+2$. Further,
$\sigma_k=\a_{k+1}b_{k+1}^2c_{k+1}-d_k=c-\frac{\a c}{k+1}-d_k$ and $s_k=b_k^2c_k^2-\a_{k+1}b_{k+1}^2c_{k+1}+\a_kb_k^2c_k +a_kb_kc_k=\frac{\a c}{k+1}+\frac{c(a-\a)}{k}+\frac{c^2}{k^2}.$
Now, take $d_k=\frac{(a+2-\a)c}{2k}\ge 0$ and an easy computation gives that there exists $k_2\ge k_1$ such that for all $k\ge k_2$ one has
$$\frac12(\mu_k+m_k+1)c_{k}-\sigma_k-s_k=-\frac{(a+2-\a)c}{2k}+\frac{2c^2-3c}{2k^2}\le0,$$
$$-\frac12(\mu_{k-1}+1)c_{k-1}+\sigma_{k-1}=\frac{c(a-2-\a)}{2(k-1)}+\frac{\a c}{k(k-1)}-\frac{c}{2(k-1)^2}\le 0$$
$$a_kb_kc_k-\sigma_k-s_k+\sigma_{k-1}=\frac{(a+2-\a)c}{2k}-\frac{(a+2-\a)c}{2(k-1)}-\frac{c^2}{k^2}\le 0.$$
Now, since by assumption  $\|x_k\|\ge \|x^*\|$ for $k\ge k_0$, we get that the right hand side of \eqref{forfratestr1} is nonpositive for all $k\ge k_2.$ Hence, for all $k\ge k_2$ we have
\begin{align}
&\mathcal{E}_{k+1}-\mathcal{E}_k+\mu_k(f_{c_k}(x_{k})-f_{c_k}(x^*))-\mu_{k-1}(f_{c_{k-1}}(x_{k-1})-f_{c_{k-1}}(x^*))+m_k(f_{c_k}(x_{k})-f_{c_k}(x^*))\nonumber \\
&+\nu_k\|x_k-x^*\|^2-\nu_{k-1}\|x_{k-1}-x^*\|^2+n_k\|x_k-x^*\|^2 +\eta_k\|x_k-x_{k-1}\|^2+b_{k-1}^2\|\n f(x_k)\|^2\le 0. \label{forfratestr2}
\end{align}
Note that $\nu_k \sim C$. Therefore, from \eqref{forfratestr2}, similarly as in the proof of Theorem \ref{convergencealgorithm}, we deduce that
$\|x_k-x^*\|$ is bounded, and therefore $(x_k)$  is bounded.
Further,
$$\lim_{k\to+\infty}(\|a(x_k-x^*)+k(x_{k+1}-x_{k})\|^2+\mu_k (f_{c_k}(x_{k})-f_{c_k}(x^*))+\nu_k\|x_k-x^*\|^2)=0,$$
that is, $\lim_{k\to+\infty}\nu_k\|x_k-x^*\|^2=0$ and hence \;
$\lim_{k\to+\infty}x_k=x^*.$

\smallskip

{\bf Case II.} Assume that there exists $k_0\in \N$ such that $\|x_k\|<\|x^*\|$ for all $k\ge k_0.$ From there we get that $(x_k)$ is bounded.  Now, take $\bar{x} \in \mathcal{H}$ a weak sequential cluster point of  $(x_k),$ which exists since   $(x_k)$ is bounded. This means that there exists a sequence $\left(k_{n}\right)_{n \in \mathbb{N}} \subseteq[k_0,+\infty)\cap \N$ such that $k_{n} \to +\infty$ and $x_{k_{n}}$ converges weakly to $\bar{x}$ as $n \to +\infty$. Since $f$ is weakly lower semicontinuous, according to Theorem \ref{convergencealgorithm} we have
$
f(\bar{x}) \leq \liminf _{n \rightarrow+\infty} f\left(x_{k_{n}}\right)=\min f \, ,$ hence $\bar{x} \in \operatorname{argmin} f.$
Since the norm is weakly lower semicontinuous, we deduce that
$$
\begin{array}{c}
\|\bar{x}\| \leq \liminf _{n \rightarrow+\infty}\left\|x_{k_{n}}\right\| \leq\left\|x^\ast \right\|.
\end{array}
$$
According to the definition of $x^\ast$, we get $\bar{x}=x^{*}.$ Therefore $(x_k)$ converges weakly to $x^\ast$. So
$$
\left\|x^\ast \right\| \leq \liminf _{k \rightarrow+\infty}\|x_k\| \leq \limsup _{t \rightarrow+\infty}\|x_k\| \leq\left\|x^\ast \right\|.$$
Therefore, we have
$ \lim _{k \rightarrow+\infty}\|x_k\|=\left\|x^\ast \right\|.$
From  the previous relation and the fact that $x_k\rightharpoonup x^\ast$ as $k \to +\infty,$ we obtain the strong convergence, that is
$
\lim _{k \rightarrow+\infty} x_k=x^\ast.$

\smallskip

{\bf Case  III.} Suppose that for every $k \geq k_{0}$ there exists $l \geq k$ such that $\left\|x^\ast \right\|>\|x_l\|$, and suppose also there exists $m \geq k$ such that $\left\|x^{*}\right\| \leq\|x_m\|$.
So, let $k_1\ge k_0$ and $l_1\ge k_1$ such that $\left\|x^\ast \right\|>\|x_{l_1}\|.$
Let $k_2>l_1$ and $l_2\ge k_2$ such that $\left\|x^\ast \right\|>\|x_{l_2}\|.$ Continuing the process,  we obtain $(x_{l_n})$, a subsequence of $(x_k)$ with the property that $\|x_{l_n}\|<\|x^*\|$ for all $n\in\N.$ By reasoning as in {\bf Case II}, we obtain that
$\lim _{n \rightarrow+\infty} x_{l_n}=x^\ast.$
Consequently, \;
$\liminf_{k \rightarrow+\infty} \|x_k-x^\ast\|=0.$
\end{proof}

\subsection{Non-smooth case}\label{non-smooth-prox}
Let us extend the results of the previous sections to the case of a proper lower semicontinuous and convex function $f: \cH \to \R \cup \left\lbrace +\infty \right\rbrace$.
We rely on the basic properties of the Moreau envelope $f_{\lambda}: \cH \to \R$  ($\lambda$ is a positive real parameter), which is defined by
\[
f_{\lambda} (x) = \min_{z \in \cH} \left\lbrace f (z) + \frac{1}{2 \lambda} \| z - x\| ^2   \right\rbrace, \quad \text{for any $x\in \cH$.}
\]
Recall that  $f_{\lambda} $ is a convex differentiable function, whose gradient is $\lambda^{-1}$-Lipschitz continuous, and such that $\min_{\cH} f= \min_{\cH} f_{\lambda}$, \; $\argmin_{\cH} f_{\lambda} = \argmin_{\cH} f$.
The interested reader may refer to \cite{BC,Bre1} for a comprehensive treatment of the Moreau envelope in a Hilbert setting. Since the set of minimizers is preserved by taking the Moreau envelope, the idea is to replace $f$ by $f_{\lambda} $ in the previous algorithm, and take advantage of the fact that $f_{\lambda} $ is continuously differentiable.  Then, algorithm ${\rm (IPATRE)}$  applied to $f_{\lambda}$ now reads (recall that $\alpha_k =1 -\frac{\alpha}{k}$)
\[
\begin{array}{l}
{\rm (IPATRE)} \quad  \left\{
\begin{array}{l}
y_k= x_k+\alpha_{k }(x_k -  x_{k-1})\\
\rule{0pt}{15pt}
x_{k+1}={ \rm prox}_{f_{\lambda}}\left( y_k - \frac{c}{k^2}x_k\right).
\end{array}
\right.
\end{array}
\]
By applying Theorems \ref{convergencealgorithm} and  \ref{strconvergencealgorithm},  we obtain fast convergence of the sequence $(x_k)$ to the element of minimum norm of $f$.
Thus, we just need to formulate these results in terms of $f$ and its proximal mapping. This is straightforward thanks to the following formulae from proximal calculus \cite{BC}:

\begin{enumerate}
\item $f_{\lambda} (x)= f(\prox_{ \lambda f}(x)) + \frac{1}{2\lambda} \|x-\prox_{\lambda f}(x)\| ^2$.
\item $\nabla f_{\lambda} (x)= \frac{1}{\lambda} \left( x-\prox_{ \lambda f}(x) \right)$.
\item $\prox_{ \theta f_{\lambda}}(x) = \frac{\lambda}{\lambda +\theta}x + \frac{\theta}{\lambda +\theta}\prox_{ (\lambda + \theta)f}(x).$
\end{enumerate}
\noindent We obtain the following relaxed inertial proximal algorithm (NS stands for non-smooth):
\begin{eqnarray*}
\begin{array}{rcl}
{\rm \mbox{(IPATRE-NS)}}  \quad
\begin{cases}
y_k=   x_{k} + ( 1- \frac{\alpha}{k}) ( x_{k}  - x_{k-1})  \\
x_{k+1} =  \frac{\lambda}{1+\lambda}\left(  y_k - \frac{c}{k^2}x_k \right) + \frac{1}{1+\lambda}\prox_{(\lambda+1) f} \left(  y_k - \frac{c}{k^2}x_k \right).
\end{cases}
\end{array}
\vspace{2mm}
\end{eqnarray*}
\begin{theorem}
Let $f: \cH \to \R \cup \left\lbrace +\infty \right\rbrace$ be a convex, lower semicontinuous, proper function.  Assume that $\a>3$. Let  $(x_k)$ be a sequence generated by
$ \mbox{\rm(IPATRE-NS)}$. Then  for all $s\in\left[\frac12,1\right[$, we have:

\smallskip

(i)  $f({\prox}_{ \lambda f}(x_k))-\min_{\cH} f=o(k^{-2s})$, $\|x_{k}-x_{k-1}\|=o(k^{-s})$,

\smallskip

\; \; $\|x_k - {\prox}_{ \lambda f}(x_k))\|=o(k^{-s})$ as $k\to+\infty.
$\vspace{2mm}

(ii)  $\ds\sum_{k=1}^{+\infty} k^{2s-1}(f({\prox}_{ \lambda f}(x_k))-\min_{\cH} f)<+\infty,$ \;  $\ds\sum_{k=1}^{+\infty} k^{2s-1}\|x_k-x_{k-1}\|^2<+\infty$,

\hspace{6mm}$\ds\sum_{k=1}^{+\infty} k^{2s}\|x_k - {\prox}_{ \lambda f}(x_k))\|^2< +\infty$.

(iii) $\liminf_{k\to+\infty}\|x_k-x^*\|=0$. Further, $(x_k)$ converges strongly to $x^*$ the element of minimum norm of
$\argmin f$, if $(x_k)$ is in the interior  of the ball $B(0,\|x^*\|)$ for $k$ large enough,  or if $(x_k)$ is in the complement of the ball $B(0,\|x^*\|)$ for $k$ large enough.
\end{theorem}

\section{Conclusion, perspective}
In the  framework of convex optimization in general Hilbert spaces, we have introduced an inertial dynamic in which the damping coefficient and the  Tikhonov regularization coefficient vanish as time tends to infinity. 
The judicious adjustment of these parameters makes it possible to obtain trajectories converging quickly (and strongly) towards the minimum norm solution.
This seems to be the first time that these two properties have been obtained for the same dynamic. Indeed, the Nesterov accelerated gradient method and the  hierarchical minimization attached to the Tikhonov regularization are fully effective within this dynamic.
On the basis of Lyapunov's analysis, we have developed an in-depth mathematical study of the dynamic which is a valuable tool for the development of corresponding results for algorithms obtained by temporal discretization. We thus obtained similar results for the corresponding proximal algorithms.
This study opens up a large field of promising research concerning  first-order optimization algorithms.
Many interesting questions such as the introduction of Hessian-driven damping to attenuate oscillations 
\cite{ACFR}, \cite{APR}, \cite{BCL}, and the study of the impact of errors, perutrbations, deserve further study.
These results also adapt well   to the numerical analysis of inverse problems for which strong convergence and obtaining a solution close to a desired state are key properties.

\if
{
\section{Numerical experiments}\label{sec:numerical}

In this section we consider two numerical experiments for the trajectories generated by the dynamical system {\rm({TRIGS})} for a convex but not strongly convex objective function
$$f:\R^2\to\R,\,\,\,f(x,y)=(ax+by)^2\mbox{ where }a,b\in\R\setminus\{0\}.$$
Observe that $\argmin f=\left\{\left(x,-\frac{a}{b}x\right):x\in\R\right\}$ and $\min f=0$. Obviously, the minimizer of minimal norm of $f$ is $x^*=(0,0).$
In the following numerical experiments the continuous time dynamical system {\rm({TRIGS})} is solved numerically with the ode45 adaptive method in MATLAB on the interval $[1, 50]$. For the Tikhonov regularization parameter we take $\e(t)=\frac{1}{t^{r}},\,r\in\{0.6,1,1.5,2\}$ and we consider the starting points $x(1)=(1,1),\,\dot{x}(1)=(-1,-1).$

In our first experiment we take $\d=2$, and we study the evolution of the two errors $\|x(t)-x^*\|$ and $f(x(t))-\min f$, for a trajectory $x(t)$ generated by the dynamical system {\rm({TRIGS})}, with different values of $a$ and $b$. So we take $a=2$ and $b=1$, values for which the function $f$ is well conditioned. The results are depicted on Figure 1, where the $y$ axis is endowed with a logarithmic scale.
\begin{figure}[hbt!]
  \centering
  \includegraphics[width=.99\linewidth]{energ1-eps-converted-to}
 \caption{Error analysis with different Tikhonov regularization parameters in the dynamical system {\rm({TRIGS})}  for a well conditioned convex objective function.}
\end{figure}\label{fig1}

Further, we consider $a=100$ and $b=0.2$, values for which the function $f$ is poorly conditioned. The other parameters remain unchanged. The results are depicted on Figure 2, where the $y$ axis is endowed with a logarithmic scale.
\begin{figure}[hbt!]
  \centering
  \includegraphics[width=.99\linewidth]{energ2-eps-converted-to}
 \caption{Error analysis with different Tikhonov regularization parameters in the dynamical system {\rm({TRIGS})}  for a poorly conditioned convex objective function.}
\end{figure}\label{fig2}

In our second experiment we are interested on the sensitivity of the generated trajectories  of the dynamical system {\rm({TRIGS})} when we change the value of $\d.$ So first we fix $\e(t)=\frac{1}{t^{0.66}},$ (note that $0.66\approx\frac{2}{3}$), and we consider $\d\in\{0.5,1,1.5,2,3\}$ for the well conditioned case $a=2,\,b=1.$ The results are depicted on Figure 3, where the y axis is endowed with a logarithmic scale.
\begin{figure}[hbt!]
  \centering
  \includegraphics[width=.99\linewidth]{energ3-eps-converted-to}
 \caption{Error analysis with different values of parameter $\d$ in the dynamical system {\rm({TRIGS})}, for $\e(t)=\frac{1}{t^{0.66}}.$}
\end{figure}\label{fig3}

Now we take the other extreme of the Tikhonov regularization parameter, namely $\e(t)=\frac{1}{t^{2}},$ and we consider the same values $\d\in\{0.5,1,1.5,2,3\}$ for the well conditioned case $a=2,\,b=1.$ The results are depicted on Figure 4, where the y axis is endowed with a logarithmic scale.
\begin{figure}[hbt!]
  \centering
  \includegraphics[width=.99\linewidth]{energ4-eps-converted-to}
 \caption{Error analysis with different values of parameter $\d$ in the dynamical system {\rm({TRIGS})}, , for $\e(t)=\frac{1}{t^{2}}.$}
\end{figure}\label{fig4}

{\bf Introducing the Hessian driven damping in (TRIGS)}

A natural extension of (TRIGS) is to add into the governing differential equation a Hessian driven damping term \cite{APR},\cite{BCL},\cite{ACFR}, that is to consider the differential equation
\begin{equation}\label{edohess}
 \mbox{(DIN)$_{\d,\b}$} \quad \quad \ddot{x}(t) + \d\sqrt{\e(t)}  \dot{x}(t) + \nabla f (x(t))+\b\n^2 f(x(t))\dot{x}(t) + \e (t) x(t) =0,
\end{equation}
where $\b>0$.
However, if we consider  the Hessian driven damping system with $\mu-$strongly convex function \cite{siegel,WRJ}, that is
$$\ddot{x}(t)+2\sqrt{\mu}\dot{x}(t)+\n f(x(t))+\b\n^2 f(x(t))\dot{x}(t)=0,$$
then  following the approach of Section 1.1  we obtain another differential equation.

Indeed, by taking the $\e$-strongly convex function $f_t(x)=f(x)+\frac{\e(t)}{2}\|x\|^2$ in the previous differential equation, and taking the general $\d$  we obtain
\begin{equation}\label{edohessnew}
\mbox{(DINnew)$_{\d,\b}$} \quad\quad \ddot{x}(t)+\left(\d\sqrt{\e(t)}+\b\e(t)\right)\dot{x}(t)+\n f(x(t))+\b\n^2 f(x(t))\dot{x}(t)+\e(t) x(t)=0,
\end{equation}
where $\b>0$.
Note that for $\b=0$ both of the above systems become (TRIGS).

In the following numerical experiment we show that indeed, by introducing the Hessian driven damping term in (TRIGS) we obtain a smoothing effect. To this end we  take $a=2$ and $b=1$, values for which the function $f$ is well conditioned and we consider $\d=2$ and $r=2$. Further we consider different values of $\beta.$ The results are depicted on Figure 5 for {(DINnew)$_{\d,\b}$} and Figure 6 for {(DIN)$_{\d,\b}$}.

\begin{figure}[hbt!]
  \centering
  \includegraphics[width=.99\linewidth]{hess1-eps-converted-to}
 \caption{Error analysis with different values of parameter $\b$ in the dynamical system \rm{(DINnew)$_{\d,\b}$}, for $\d=2$ and $\e(t)=\frac{1}{t^{2}}.$}
\end{figure}\label{fig5}

\begin{figure}[hbt!]
  \centering
  \includegraphics[width=.99\linewidth]{hess2-eps-converted-to}
 \caption{Error analysis with different values of parameter $\b$ in the dynamical system \rm{(DIN)$_{\d,\b}$}, for $\d=2$ and $\e(t)=\frac{1}{t^{2}}.$}
\end{figure}\label{fig6}
We conclude that \rm{(DINnew)$_{\d,\b}$} has a better behaviour than \rm{(DIN)$_{\d,\b}$} and indeed both systems generate trajectories with a smoothing effect.

}
\fi

\end{document}